\documentclass[final,3p,times]{elsarticle}
\usepackage[english]{babel}
\usepackage[utf8x]{inputenc}
\usepackage[T1]{fontenc}
\usepackage{amssymb}
\usepackage{amsthm}
\usepackage{amsmath}
\usepackage{mathrsfs}
\usepackage{hyperref}
\hypersetup{
    colorlinks,
    linkcolor=red,
    citecolor=black,
    urlcolor=blue
}
\usepackage{floatrow}
\usepackage{color}
\biboptions{numbers,sort&compress}
\usepackage{enumitem}
\usepackage{tikz}
\usepackage{graphicx}
\usetikzlibrary{shapes}
\usetikzlibrary{arrows}
\usepackage{tkz-graph}
\usepackage{caption}
\usepackage{subcaption}
\usepackage{wrapfig}
\usepackage{xparse}

\newtheorem{thm}{Theorem}[section]
\newtheorem{lem}[thm]{Lemma}

\newtheorem{conj}[thm]{Conjecture}

\newtheorem{defn}[thm]{Definition}

\newtheorem{cor}[thm]{Corollary}

\newtheorem{que}[thm]{Question}
\newtheorem{rem}[thm]{Remark}

\ExplSyntaxOn
\NewDocumentCommand{\breakdot}{m}
 {
  \tl_set:Nn \l_tmpa_tl { #1 }
  \tl_replace_all:Nnn \l_tmpa_tl { . } { .\discretionary{}{}{} }
  \tl_use:N \l_tmpa_tl
 }
\ExplSyntaxOff

\journal{European Journal of Combinatorics}
\begin{document}
\pagestyle{plain}
\begin{frontmatter}



\title{Strong Ramsey Games in Unbounded Time}


\author[1]{Stefan David}
\ead{sd637@cam.ac.uk}
\author[2]{Ivailo Hartarsky\corref{cor}
\footnote{Present address: CEREMADE, CNRS, UMR 7534, Universit\'e Paris-Dauphine, PSL University, Place du Mar\'echal de Lattre de Tassigny, 75775 Paris Cedex 16, France}}
\ead{hartarsky@ceremade.dauphine.fr}
\cortext[cor]{Corresponding author}
\author[1]{Marius Tiba}
\ead{mt576@cam.ac.uk}
\address[1]{DPMMS, University of Cambridge, Wilberforce Road, Cambridge, UK CB3 0WA}
\address[2]{DMA UMR 8553, \'Ecole Normale Sup\'erieure, CNRS, PSL University, 45 rue d'Ulm, 75005 Paris, France}

\begin{abstract}
For two graphs $B$ and $H$ the strong Ramsey game $\mathcal{R}(B,H)$ on the board $B$ and with target $H$ is played as follows. Two players alternately claim edges of $B$. The first player to build a copy of $H$ wins. If none of the players win, the game is declared a draw. A notorious open question of Beck \cite{Beck96,Beck02,Beck08} asks whether the first player has a winning strategy in $\mathcal{R}(K_n,K_k)$ in bounded time as $n\rightarrow\infty$. Surprisingly, in a recent paper \cite{Hefetz17} Hefetz, Kusch, Narins, Pokrovskiy, Requil\'e and Sarid constructed a $5$-uniform hypergraph $\mathcal{H}$ for which they proved that the first player does not have a winning strategy in $\mathcal{R}(K_n^{(5)},\mathcal{H})$ in bounded time. They naturally ask whether the same result holds for graphs. In this paper we make further progress in decreasing the rank.

In our first result, we construct a graph $G$ (in fact $G=K_6\setminus K_4$) and prove that the first player does not have a winning strategy in $\mathcal{R}(K_n \sqcup K_n,G)$ in bounded time. As an application of this result we deduce our second result in which we construct a $4$-uniform hypergraph $G'$ and prove that the first player does not have a winning strategy in $\mathcal{R}(K_n^{(4)},G')$ in bounded time. This improves the result in the paper above.

By compactness, an equivalent formulation of our first result is that the game $\mathcal{R}(K_\omega\sqcup K_\omega,G)$ is a draw. Another reason for interest on the board $K_\omega\sqcup K_\omega$ is a folklore result that the disjoint union of two finite positional games both of which are first player wins is also a first player win. An amusing corollary of our first result is that at least one of the following two natural statements is false: (1) for every graph $H$, $\mathcal{R}(K_\omega,H)$ is a first player win; (2) for every graph $H$ if $\mathcal{R}(K_\omega,H)$ is a first player win, then $\mathcal{R}(K_\omega\sqcup K_\omega,H)$ is also a first player win. Surprisingly, we cannot decide between the two.

\end{abstract}

\begin{keyword}
Positional games\sep Ramsey games\sep Strong games


\MSC[2010] 05C57 \sep 91A43 \sep 91A46

\end{keyword}

\end{frontmatter}
\section{Introduction}
\subsection{Background}
Positional games were first studied by Hales and Jewett \cite{Hales63} and Erd\H{o}s and Selfridge \cite{Erdos73}. The general setting was given by Berge \cite{Berge76}, but the field was shaped by the numerous works of Beck since the early 80s including \cite{Beck82b,Beck81,Beck82a,Beck85,Beck96,Beck02,Beck08}. Though the most natural positional games are the strong ones, in which both players compete to achieve the same objective, the theory has largely deviated from this direction due to the prohibitive difficulty of strong games. Many weak variants have been developed and have proved more suitable for study, like Maker-Breaker games, for which much is known (see e.g. \cite{Beck08}). However, the strong games remain very poorly understood.

In the present paper we study instances of a particular type of strong game -- the strong Ramsey game. This was first introduced by Harary \cite{Harary72} for cliques and later for arbitrary graphs \cite{Harary82}. In the strong Ramsey game, or simply the \emph{Ramsey game}, $\mathcal{R}(B,H)$ on a finite or infinite graph $B$, called the \emph{board}, with \emph{target graph} $H$, two players, \textbf{P1} and \textbf{P2}, alternate to take previously unclaimed edges of $B$, starting with \textbf{P1}. The first player to claim an isomorphic copy of $H$ \emph{wins} and if none does so in finite time, the game is declared a \emph{draw}. More generally, one might take $B$ and $H$ to be $r$-uniform hypergraphs. This setting was already mentioned by Beck and Csirmaz and Beck in \cite{Beck81,Beck82b}.

For fixed graphs $B, H$, in $\mathcal{R}(B,H)$, a very general strategy stealing argument due to Nash shows that \textbf{P2} cannot have a winning strategy. Moreover, for any fixed target graph $H$ it follows from Ramsey's theorem \cite{Ramsey30} that $\mathcal{R}(K_n,H)$ cannot end in a draw for $n$ sufficiently large. Therefore, in $\mathcal{R}(K_n,H)$ \textbf{P1} has a winning strategy for $n$ sufficiently large. The strategy given by this argument is not explicit, and indeed almost no examples of explicit strategies are known. In particular, no explicit strategy has been exhibited for $\mathcal{R}(K_n,K_k)$ with $k\geq 5$ and $n$ large.

This fact makes it difficult to attack most natural questions in the field. For instance a notorious open question popularised by Beck \cite{Beck96,Beck02,Beck08} asks if, for fixed $k$, in the game $\mathcal{R}(K_n,K_k)$, \textbf{P1} has a winning strategy in bounded time as $n \rightarrow \infty$. An easy compactness argument shows that this is equivalent to the game $\mathcal{R}(K_\omega,K_k)$ being a \textbf{P1}-win. The answer is conjectured to be in the affirmative \cite{Beck96,Beck02,Pekec96} but no progress has been made on the problem for $k\geq 5$. Beck emphasised the importance of the question when he listed the question among his ``7 most humiliating open problems'' \cite{Beck08}.
In the opposite direction, another notorious open question asks if for every fixed target graph $H$ in the game $\mathcal{R}(K_n,H)$, \textbf{P1} has a winning strategy in bounded time as $n \rightarrow \infty$ or, equivalently, $\mathcal{R}(K_\omega,H)$ is a \textbf{P1} win. In a recent paper \cite{Hefetz17}, Hefetz, Kusch, Narins, Pokrovskiy, Requilé and Sarid addressed the natural generalisation to hypergraphs, and changed the intuition about this phenomenon completely. Surprisingly, in  \cite{Hefetz17} they exhibited a target $5$-uniform hypergraph $\mathcal{H}$ for which they constructed an explicit drawing strategy for \textbf{P2} in the game $\mathcal{R}\left(K_{\omega}^{(5)},\mathcal{H}\right)$. This result provides strong evidence that just strategy stealing and Ramsey-type arguments are insufficient to attack Beck's conjecture.

However, the corresponding question for graphs, as asked in \cite{Hefetz17}, still remains open. As we explain in Section~\ref{sec:overview}, as the rank decreases this question becomes much harder. In this paper we make further progress in decreasing it. In our first result, Theorem \ref{th:main} (see Corollary~\ref{cor:main}), we exhibit a graph $G$ (in fact $G=K_6\setminus K_4$, see Figure~\ref{fig:G}), for which we prove that in $\mathcal{R}(K_n\sqcup K_n, G)$ \textbf{P1} does not have a winning strategy in bounded time. In order to do so we build on the work of~\cite{Hefetz17}. However, there is a serious obstacle in adapting the strategy developed in~\cite{Hefetz17} from $5$-uniform hypergraphs to graphs. As we further explain in Section~\ref{sec:overview}, although the strategy developed in~\cite{Hefetz17} is a strong game drawing strategy, the core of the argument is a weak game fast winning strategy that is inapplicable to graphs. Therefore, most of our effort is spent developing a much more intricate and elaborate strategy for graphs reflecting characteristic difficulties of strong games, which we discuss further in Section~\ref{sec:remarks}.

Turning to hypergraphs, in our second result, Theorem~\ref{th:hyper} (see Corollary~\ref{cor:hyper}), we exhibit a $4$-uniform hypergraph $G'$, obtained by adding $2$ new vertices and including them in all edges of $G$, for which we prove that in $\mathcal{R}\left(K_n^{(4)},G'\right)$ \textbf{P1} does not have a winning strategy in bounded time. In order to do so we cover the board $K_n^{(4)}$ with copies of the board $K_{n-2}^{(2)}$, noting that $K_n^{(4)}= \bigcup_{X\neq Y\in V(K_n^{(4)})}K_{n-2}^{X,Y}$, where $K_{n-2}^{X,Y}$ is the set of hyperedges containing $X$ and $Y$, naturally identified with $K_{n-2}$. While in the graph setting the strategy is considerably more involved, in the hypergraph setting the strategy is much more difficult to analyse.

It is not clear whether one should expect $\mathcal{R}(K_n,G)$ to admit a bounded time winning strategy for \textbf{P1}, but it is known (see~\cite{Bowler12} for a simple proof) that for having a winning strategy on $K_n$ is equivalent to having one on $K_n\sqcup K_n$. Thus, by compactness, our result refutes one of two natural conjectures extending simple finite board facts to infinite boards. Namely, it is not the case that on the one hand for every graph $H$ $\mathcal{R}(K_\omega,H)$ is a \textbf{P1}-win and on the other hand for every graph $H$ if $\mathcal{R}(K_\omega,H)$ is a \textbf{P1}-win, then $\mathcal{R}(K_\omega\sqcup K_\omega,H)$ is also a \textbf{P1}-win.

\subsection{Results}
In this section we state our results. 

\begin{defn}
Let $G=K_6\setminus K_4$ be the graph with edge set $E(K_6)\setminus E(K_4)$ (see Figure~\ref{fig:G}).
\end{defn}
In the graph setting the central result is as follows. 
\begin{thm}
\label{th:main}
The game $\mathcal{R}(K_\omega\sqcup K_\omega,G)$ is a draw.
\end{thm}
As mentioned in the previous section, one can also formulate the result for finite boards. The proof implies the following.
\begin{cor}
\label{cor:main}
In the game $\mathcal{R}(K_n\sqcup K_n,G)$ \textbf{P1} does not have a strategy which guarantees him a win in less than $2n-O(1)$ total moves, as $n\rightarrow\infty$.
\end{cor}
\begin{rem}
Note that we obtain an explicit drawing strategy for the first player in $\mathcal{R}(K_\omega,G)$.
\end{rem}
An amusing corollary of Theorem~\ref{th:main} is the following.
\begin{cor}
It is not the case that on the one hand for every graph $H$ $\mathcal{R}(K_\omega,H)$ is a \textbf{P1}-win and on the other hand for every graph $H$ if $\mathcal{R}(K_\omega,H)$ is a \textbf{P1}-win, then $\mathcal{R}(K_\omega\sqcup K_\omega,H)$ is also a \textbf{P1}-win.
\end{cor}

In the hypergraph setting we will work with the following hypergraph.
\begin{defn}
\label{def:G'}
Let $G'$ be the $4$-uniform hypergraph obtained from our graph $G$ as $V(G')=V(G)\sqcup \{X,Y\}$ and $E(G')=\{e\sqcup\{X,Y\},e\in E(G)\}$.
\end{defn}

Below is the main result on hypergraphs.
\begin{thm}
\label{th:hyper}
The game $\mathcal{R}\left(K_{\omega}^{(4)},G'\right)$ is a draw.
\end{thm}

As before, the proof implies the following version of the result for finite boards.
\begin{cor}
\label{cor:hyper}
In the game $\mathcal{R}\left(K_n^{(4)},G'\right)$ \textbf{P1} does not have a strategy which guarantees him a win in less than $2n-O(1)$ total moves, as $n\rightarrow\infty$.
\end{cor}

\subsection{Overview}
\label{sec:overview}

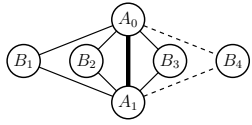
\begin{figure}  
\begin{center}\scalebox{.55}{
\begin{tikzpicture}
  \SetGraphUnit{3}
\GraphInit[vstyle=Normal]
  \SetVertexNormal[Shape      = circle,
                   LineWidth  = 1pt]
  \SetVertexMath
  \Vertex[L=B_2]{C}
  \Vertex[x=1,y=1,L=A_0]{B}
  \Vertex[x=1,y=-1,L=A_1]{A}
  \Vertex[x=2,y=0,L=B_3]{D}
  \Vertex[x=-1.5,y=0,L=B_1]{E}
  \Vertex[x=3.5,y=0,L=B_4]{F}
  
  \Edge(A)(C)
  \Edge(A)(D)
  \Edge(B)(C)
  \Edge(B)(D)
  \Edge(A)(E)
  \Edge(B)(E)
  \SetUpEdge[style={dashed}]
  \Edge(A)(F)
  \Edge(B)(F)
  \SetUpEdge[lw = 3pt]
  \Edge(A)(B)
\end{tikzpicture}}
\end{center}
\caption{The graph $G$. The dashed lines form the pair, the solid lines form the core and the base is thickened.}
 \label{fig:G}
\end{figure}

\paragraph{Graph setting}
Throughout the paper, unless otherwise stated, all configurations are considered after \textbf{P1}'s move and before \textbf{P2}'s move. The graph $G$ is formed by a \emph{base} and four \emph{pairs} of edges connected to it (see Figure~\ref{fig:G}). We decompose the graph $G$ into a pair and a \emph{core} which is formed by the base and the other three pairs.

We construct the drawing strategy for \textbf{P2} in the game $\mathcal{R}(K_\omega\sqcup K_\omega,G)$ in three stages. We denote the copy of $K_\omega$ in which \textbf{P1} takes his first edge by $K^1$ and the other copy by $K^2$.

In the first stage \textbf{P2} builds a core in $K^2$. While doing this \textbf{P2} ensures that the following two statements hold. On the one hand, during the entire stage \textbf{P2} remains ahead of \textbf{P1} in building $G$ in $K^2$, which is intuitively possible, since \textbf{P2} is the first player in $K^2$. In particular, at the end of the first stage \textbf{P1} does not have a threat in $K^2$. On the other hand, at the end of the stage \textbf{P1} has at most $|E(G)|-1$ edges in $K^1$. At the beginning of the second stage \textbf{P2} checks if \textbf{P1} has a threat in $K^1$. If this is the case, \textbf{P2} blocks \textbf{P1}'s (possibly infinite) threats as long as \textbf{P1} keeps making new ones in $K^1$. The graph $G$ is such that \textbf{P1} cannot force a win by making such consecutive threats. If at some point \textbf{P1} does not have a threat in $K^1$, then, in the third stage, \textbf{P2} aims to build the pair in $K^2$  and complete $G$. He does so by making a (possibly infinite) series of threats from a well-chosen endpoint of the base to new vertices. The choice is such that \textbf{P1}'s responses to \textbf{P2}'s threats cannot be part of a threat of \textbf{P1}.

Roughly speaking, in~\cite{Hefetz17} the strategy is divided into the same three stages: building a core, blocking threats and then making threats. However, there is a severe obstruction in transferring the first stage of the strategy from $5$-uniform hypergraphs to graphs. Indeed, in~\cite{Hefetz17} the target hypergraph has a non-trivial and identifiable core that admits a weak game fast building strategy (with number of moves equal to the size of the core). The construction of such a core relies heavily on the high rank. Unfortunately, for graphs this is not feasible anymore because the construction of any non-trivial identifiable core can be delayed by \textbf{P1}. Therefore, in the graph setting one should abandon such weak game strategy and face the essential difficulty of strong game strategies, by allowing delay and making sure to block the other player while building one's own target graph.

\paragraph{Hypergraph setting}
As it was discussed in the introduction, the board $K_\omega^{(4)}$ can be viewed as 
$\bigcup_{X,Y\in V\left(K_\omega^{(4)}\right)}K_\omega^{X,Y}$,
where $K_\omega^{X,Y}=\{e\in E(K_\omega^{(4)}), X,Y\in e\}$ identifies naturally with the board $K_\omega$. Note that a copy of $G'$ is contained in exactly one of the boards $K_\omega^{X,Y}$ and identifies with a graph $G$ in this board.

The drawing strategy of \textbf{P2} in the game $\mathcal{R}(K_\omega^{(4)},G')$ still follows the same three stages. In the first stage \textbf{P2} builds a core in a board $K_\omega^{X,Y}$ (corresponding to $K^2$ from the graph setting) disjoint from \textbf{P1}'s first hyperedge by using a very simplified version of the first stage of the strategy for $\mathcal{R}(K_\omega\sqcup K_\omega, G)$, which exploits the larger number of symmetries in higher rank. \textbf{P2} ensures that at the end of the first stage \textbf{P1} does not have a threat in any board $K_\omega^{X,Y}$, except possibly exactly one of the six boards containing \textbf{P1}'s first hyperedge. The second and third stages remain unchanged. However, in the hypergraph setting we need a more subtle analysis than in the graph setting, as the different boards intersect. The key observation is that for $\{X_1, Y_1\} \cap \{X_2, Y_2\} = \varnothing$ the boards $K_\omega^{X_1,Y_1}$ and $K_\omega^{X_2,Y_2}$ share only one hyperedge. In this respect it is clear that higher rank simplifies the problem.

\subsection{Organisation of the paper}
The rest of the paper is organised as follows. In Section~\ref{sec3} we establish the key lemmas which are used repeatedly in the proof. In Section~\ref{sec4} we construct a drawing strategy for \textbf{P2} in $\mathcal{R}(K_\omega\sqcup K_\omega,G)$. In Section~\ref{sechyp} we derive a drawing strategy for \textbf{P2} in $\mathcal{R}\left(K_\omega^{(4)},G\right)$. In Section~\ref{sec:remarks} we make some concluding remarks and state some open questions.

\section{Preliminaries}
\label{sec3}
\subsection{Setup and notation}

For this section and the next one we consider the game $\mathcal{R}(K_\omega\sqcup K_\omega,G)$, where the graph $G$ is defined below. Denote the two disjoint copies of $K_\omega$ constituting the board by $K^1$ and $K^2$ and assume without loss of generality that \textbf{P1}'s first edge is taken in $K^1$.

For technical reasons, we consider the version of the game in which \textbf{P1} stops playing after a finite number of moves, and the game continues until \textbf{P2} builds a copy of $G$. In order to show that \textbf{P2} has a drawing strategy, it is enough to show that at the end \textbf{P1} does not have a copy of $G$.

\begin{defn}
{
    \def\OldComma{,}
    \catcode`\,=13
    \def,{%
      \ifmmode%
        \OldComma\discretionary{}{}{}%
      \else%
        \OldComma%
      \fi%
    }%

Let $G=(V, E)$ be the graph with vertex set $V= \{ A_0, A_1, B_1, B_2, B_3, B_4 \}$ and edge set $E= \{A_0A_1, A_0B_i, A_1B_i \mid i=1, \hdots, 4 \}$ (see Figure~\ref{fig:G}). With a slight abuse of notation we shall refer to any isomorphic copy of $G$ as $G$.

We call the edge $A_0A_1$ the \emph{base} of $G$. Note that any automorphism of $G$ fixes the base, so it is well defined.
}
\end{defn}

For the following set of definitions we consider the game at a certain stage. Recall from Section~\ref{sec:overview} that all configurations are considered after \textbf{P1}'s move and before \textbf{P2}'s move.

\begin{defn}
We say that a given graph $G$ is \textbf{P1}-\emph{free} if \textbf{P1} has no edges in $G$ and we say it is \textbf{P2}-\emph{free} if \textbf{P2} has no edges in $G$.
\end{defn}

\begin{defn}
For a given graph $G$, let $e_{\textbf{P1}}(G)$ be the number of edges in $G$ taken by \textbf{P1} if $G$ is \textbf{P2}-free; otherwise, define it to be $0$. Define similarly $e_{\textbf{P2}}(G)$. For a given vertex $A$ let $\deg_{\textbf{P1}}(A)$ and $\deg_{\textbf{P2}}(A)$ be the number of edges that contain $A$ taken by \textbf{P1} and \textbf{P2}, respectively. 
\end{defn}
Note that in the course of the game for a given $G$ the quantities $e_{\textbf{P1}}(G)$ and $e_{\textbf{P2}}(G)$ may decrease, becoming $0$.

\begin{defn}
We say that a vertex $F$ is \textbf{P1}-\emph{free} if $\deg_{\textbf{P1}}(F)=0$, we say it is \textbf{P2}-\emph{free} if $\deg_{\textbf{P2}}(F)=0$ and we say it is \emph{free} if $\deg_{\textbf{P1}}(F)=\deg_{\textbf{P2}}(F)=0$.
\end{defn}

\begin{defn}
\label{def:potential}
We call an edge $A_0A_1$ taken by \textbf{P2} in $K^2$ a \emph{potential base} if there exist two vertices $B_1,B_2$ in $K^2$ such that \textbf{P2} has the edges $A_0B_1$, $A_0B_2$, $A_1B_1$ and $A_1B_2$ and there exists a special vertex $X \in \{ A_0, A_1 \}$ such that:
\begin{enumerate}[label=(\roman*)]
\item \label{cond:i}\textbf{P1} does not have a triangle $X,T_1,T_2$;
\item \label{cond:ii} \textbf{P1} does not have a $4$-cycle $X, C_1, C_2, C_3$ with the edge $X C_2$ not taken by \textbf{P2}.
\end{enumerate}
\end{defn}

\begin{rem}
\label{rem:conda}
If the edge $A_0A_1$ is a potential base with special vertex $X$, then if later in the game \textbf{P1} constructs an additional star $X F_1, X F_2, \hdots$ from $X$ to \textbf{P1}-free vertices $F_1, F_2, \hdots$ together with exactly $r$ extra edges, then he has at most $r$ triangles sharing a common edge that contain $X$.
\end{rem}

\subsection{Lemmas}
In this section we shall present the main technical result, Lemma~\ref{lem:2}, which states that if the game has reached certain configurations, then \textbf{P2} has a drawing strategy. This result will be used recurrently in conjunction with the strategy constructed in Section~\ref{sec4}. In order to check one of the hypotheses of Lemma~\ref{lem:2} we also establish Lemma~\ref{lem:3}.

\begin{lem}\label{lem:2}
Assume that before \textbf{P2}'s turn, the game has the following properties:
\begin{enumerate}[label=(\alph*)]
 \item \label{cond:c}\label{cond:bi} \textbf{P1} has at most six edges in $K^1$;
 \item \label{cond:bii}\label{cond:b} for any $G$ in $K^2$, $e_{\textbf{P1}}(G) \leq 5$;
 \item\label{cond:a} \textbf{P2} has a potential base $A_0A_1$.
\end{enumerate}
Then, \textbf{P2} has a drawing strategy.
\end{lem}

\begin{proof}
Let $\mathcal{E}^1_0$ and $\mathcal{E}^2_0$ be the set of initial edges of \textbf{P1} in $K^1$ and $K^2$ respectively. Without loss of generality let $A_0$ be the special vertex of the potential base $A_0A_1$. The line of play of \textbf{P2} is divided in three stages as follows. In the first stage, \textbf{P2} takes edges from $A_1$ to free vertices $F_i$ for $i=1,\ldots,k$ until at some point \textbf{P1} does not take the edge $A_0F_k$ (this necessarily happens, since \textbf{P1} makes only a finite number of moves). Then \textbf{P2} takes $A_0F_k$. Let $\mathcal{E}^2_1=\{A_0F_i,1\le i<k\}$ be the star taken by \textbf{P1} and $\mathcal{E}_2$ be the set consisting of the last two edges taken by \textbf{P1} (if he did not stop playing).\\ 

\noindent\textbf{Claim.} One can guarantee that after a finite number of moves, called second stage, the configuration (before \textbf{P2}'s move) satisfies the following two properties.
\begin{itemize}
    \item For each $G$ in $K^1$ we have that $e_\textbf{P1}(G)\leq 7$.
    \item The set of edges \textbf{P1} has taken in $K^2$ is $\mathcal{E}_0^2\sqcup\mathcal{E}_1^2\sqcup\mathcal{E}^2$ for some $\mathcal{E}^2$ with $|\mathcal{E}^2|\le 2$.
\end{itemize}
\begin{proof}[Proof of the claim]
We consider three cases for the configuration at the end of the first stage (before \textbf{P2}'s move).\\

\noindent\textbf{Case I.} \textbf{P1} has exactly $8$ edges in $K^1$ forming a copy $G^1$ of $G$ with its base $C_0C_1$ present but the edge $C_0D_1$ absent.

As $|\mathcal{E}_0^1| \le 6$, we have that $\mathcal{E}_2$ is contained in $G^1$. \textbf{P2} takes the edge $C_0D_1$ and, while \textbf{P1} keeps taking edges $C_{\epsilon_n}D_n$ to \textbf{P1}-free vertices $D_n$, \textbf{P2} keeps taking the edges $C_{1-\epsilon_n}D_n$ for $\epsilon_n \in \{0,1\}$. Eventually \textbf{P1} either stops or takes a different kind of edge $\varepsilon$. Let $\mathcal{E}=\{\varepsilon\}$. Note that $e_\textbf{P1}(G)\leq 7$ for all $G$ in $K^1$. Indeed, this immediately follows from the fact that for all vertices $C \in V(K^1)\setminus \{C_0, C_1\}$ we have that $\deg_\textbf{P1}(C) \le 3$.\\

\noindent\textbf{Case II.} \textbf{P1} has exactly $8$ edges in $K^1$ forming a copy $G^1$ of $G$ without its base $C_0C_1$.

As $|\mathcal{E}_0^1| \le 6$, we have that $\mathcal{E}_2$ is contained in $G^1$. \textbf{P2} takes $C_0C_1$ and then \textbf{P1} takes an edge $\varepsilon$. Let $\mathcal{E}=\{\varepsilon\}$. Note that $e_\textbf{P1}(G)\le 7$ for all $G$ in $K^1$.\\

\noindent\textbf{Case III.} For any $G$ in $K^1$, $e_{\textbf{P1}}(G)\le 7$.

In this case we directly skip to the third stage and set $\mathcal{E}=\mathcal{E}_2$.\\

In all cases, let $\mathcal{E}^2=\mathcal{E} \cap K^2$ and note that the set of edges \textbf{P1} has taken in $K^2$ is $\mathcal{E}_0^2\sqcup\mathcal{E}_1^2\sqcup\mathcal{E}^2$, with $|\mathcal{E}^2| \le 2$, as claimed.
\end{proof}
In the third stage \textbf{P2} takes edges from $A_1$ to free vertices $F_{i}$ for $k<i\le l$ until at some point \textbf{P1} does not take the edge $A_0F_{l}$. Then \textbf{P2} takes $A_0F_l$ and constructs a $G$. The game stops after \textbf{P2}'s move. Let $\mathcal{E}^2_3=\{A_0F_i,k<i<l\}$ and $\mathcal E_4$ be the set consisting of the last edge taken by \textbf{P1} (if he did not stop playing).\\ 

\noindent\textbf{Claim. }At the end of the third stage \textbf{P1} does not have a $G$.\\

\begin{proof}[Proof of the claim]
Notice that by the first Claim for any $G$ in $K^1$ we have $e_\textbf{P1}(G)\le 7+1$, since $|\mathcal E_4|\le 1$. We claim that the same holds in $K^2$. Indeed, \textbf{P1} cannot have a $G$ in $K^2$ that does not intersect $\mathcal{E}_1^2\cup\mathcal{E}_3^2$, as $e_{\textbf{P1}}(G) \leq 5+3$ by \ref{cond:bii} and $|\mathcal E_4|+|\mathcal E^2|\le 3$. On the other hand, \textbf{P1} cannot have a $G$ in $K^2$ that contains some $A_0F_i$, as neither of the vertices $A_0$ and $F_i$ can be in the base. This is because \textbf{P1} has at most $3$ triangles sharing an edge that contain vertex $A_0$ by Remark~\ref{rem:conda}, and $\deg_{\textbf{P1}}(F_i) \le 1+3$.
\end{proof}
Clearly, the last claim implies that \textbf{P2} has a drawing strategy.
\end{proof}

Note that \textbf{P2} can also adopt a slightly different strategy in the first stage. Say that \textbf{P2} has the triangle $A_0A_1B_1$ and assume that \textbf{P2} takes the edges from $A_1$ to free vertices $F_i$ for $i=1,\dots,5$, while \textbf{P1} takes the edges $A_0F_i$. Then \textbf{P2} can win in three moves by taking three of the five edges $B_1F_i$. Thus, in the original version of the game $\mathcal{R}(K_\omega\sqcup K_\omega,G)$ (in which \textbf{P1} does not stop playing after a finite time) \textbf{P2} can always build the core of $G$ in his drawing strategy, as claimed in Section~\ref{sec:overview}.

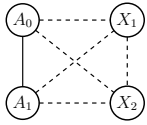
\begin{figure}
\begin{center}\scalebox{.55}{
\begin{tikzpicture}
  \SetGraphUnit{3}
\GraphInit[vstyle=Normal]
  \SetVertexNormal[Shape      = circle,
                   LineWidth  = 1pt]
  \SetVertexMath
  \Vertex[x=1,y=1,L=A_0]{A}
  \Vertex[x=1,y=-1,L=A_1]{B}
  \Vertex[x=3.5,y=-1,L=X_2]{Y}
  \Vertex[x=3.5,y=1,L=X_1]{X}
  \SetVertexNoLabel
  
  \Edge(A)(B)
  \SetUpEdge[style={dashed}]
  \Edge(A)(X)
  \Edge(B)(Y)
  \Edge(X)(Y)
  \Edge(A)(Y)
  \Edge(B)(X)
\end{tikzpicture}}
\end{center}
\vspace{-0.5cm}
\caption{The $2\Delta$-configuration.}
\label{fig:lem3}
\end{figure}
We next provide a quick way to check that condition \ref{cond:a} of Lemma~\ref{lem:2} holds.

\begin{defn}
At a certain stage of the game, and for a given edge $A_0A_1$ taken by \textbf{P2} in $K^2$, we call an edge taken by \textbf{P1} \emph{good for $A_0A_1$} if every $4$-cycle $A_i,C_1,C_2,C_3$ with $A_iC_2$ not taken by \textbf{P2} and every triangle $A_i,C_1,C_2$ that contain this edge also contains an edge taken by \textbf{P2}. In particular, all edges of \textbf{P1} in $K^1$ are good for any edge of \textbf{P2} in $K^2$. We call \emph{bad for $A_0A_1$} an edge of \textbf{P1} in $K^2$ that is not good for $A_0 A_1$.
\end{defn}
Equivalently, an edge $X_1X_2$ is good for $A_0A_1$ if and only if it is in $K^1$ or
\begin{equation}
\label{rem35}\{X_1, X_2\}\cap\{A_0, A_1 \}=\varnothing, \qquad \text{ \textbf{P2} has $A_0X_1$ or $A_0X_2$},\qquad\text{ \textbf{P2} has $A_1X_1$ or $A_1X_2$}.\tag{*}
\end{equation}
We call a $2\Delta$-configuration for an edge $A_0A_1$ the edges $A_0X_1, A_0X_2, A_1X_1, A_1X_2,  X_1X_2$ (see Figure~\ref{fig:lem3}) for any two vertices $X_1,X_2$ disjoint from $A_0,A_1$.
\begin{lem}
\label{lem:3}
Assume that before his turn \textbf{P2} has an edge $A_0A_1$ in $K^2$ such that:
\begin{enumerate}[label=(\alph*)]
 \item\label{cond:5bad} \textbf{P1} has at most $5$ bad edges for $A_0A_1$;
 \item\label{cond:2d} \textbf{P1} does not have a $2\Delta$-configuration for $A_0A_1$; 
 \item\label{cond:base} \textbf{P2} has edges $A_0B_1$, $A_0B_2$, $A_1B_1$, $A_1B_2$ for some vertices $B_1,B_2$.
\end{enumerate}
Then $A_0A_1$ is a potential base i.e. condition \ref{cond:a} of Lemma~\ref{lem:2} holds.
\end{lem}

\begin{proof}
Assume for a contradiction that $A_0A_1$ is not a potential-base. Let us call a triangle or $4$-cycle as in Definition~\ref{def:potential} an \emph{obstruction} for $X$ ($X\in\{A_0,A_1\}$). By condition \ref{cond:base}, \textbf{P1} has an obstruction for each of $A_0$ and $A_1$. Yet, since \textbf{P2} has $A_0A_1$, a single obstruction cannot contain both $A_0$ and $A_1$. So \textbf{P1} has two separate obstructions, one containing $A_0$ but not $A_1$, and one containing $A_1$ but not $A_0$. The obstruction containing $A_0$ has degree $2$ at $A_0$, so the obstruction containing $A_1$ has at most $5-2$ edges by condition~\ref{cond:5bad}, since all edges in an obstruction are bad. Thus, both obstructions are triangles, they contain $A_0$ and $A_1$ and they share a common edge by condition \ref{cond:5bad}, which contradicts condition~\ref{cond:2d}, as $A_0A_1$ is taken by \textbf{P2}.
\end{proof}

\section{A drawing strategy}\label{sec4}
\subsection{Setting}
\begin{figure}
\begin{center}\scalebox{.55}{
\begin{tikzpicture}
  \SetGraphUnit{3}
\GraphInit[vstyle=Normal]
  \SetVertexNormal[Shape      = circle,
                   LineWidth  = 1pt]
  \SetVertexMath
  \Vertex[x=-8,y=-1,L=1]{A}
  \Vertex[x=8,y=-1,L=2]{B}
  \Vertex[x=-6,y=-2,L=2]{A2}
  \Vertex[x=-11,y=-3,L=1]{A11}
  \Vertex[x=-9,y=-3,L=2]{A12}
  \Vertex[x=-7,y=-3,L=1]{A21}
  \Vertex[x=-5,y=-3,L=2]{A22}
  \Vertex[x=3,y=-2,L=1]{B1}
  \Vertex[x=13,y=-2,L=2]{B2}
  \Vertex[x=9,y=-3,L=2]{B12}
  \Vertex[x=12,y=-3,L=1]{B21}
  \Vertex[x=14,y=-3,L=2]{B22}
  \Vertex[x=2,y=-4,L=2]{B112}
  \Vertex[x=7,y=-4,L=1]{B121}
  \Vertex[x=11,y=-4,L=2]{B122}
  \Vertex[x=-8,y=-5,L=1]{B1111}
  \Vertex[x=-4,y=-5,L=2]{B1112}
  \Vertex[x=0,y=-5,L=1]{B1121}
  \Vertex[x=4,y=-5,L=2]{B1122}
  \Vertex[x=-9,y=-6,L=1]{B11111}
  \Vertex[x=-7,y=-6,L=2]{B11112}
  \Vertex[x=-5,y=-6,L=1]{B11121}
  \Vertex[x=-3,y=-6,L=2]{B11122}
  \Vertex[x=2,y=-6,L=2]{B11212}
  \Vertex[x=4,y=-6,L=1]{B12111}
  \Vertex[x=6,y=-6,L=2]{B12112}
  \Vertex[x=10,y=-6,L=2]{B12122}
  \Vertex[x=-6,y=-7,L=1]{B111211}
  \Vertex[x=-4,y=-7,L=2]{B111212}
  \Vertex[x=-3,y=-7,L=1]{B112111}
  \Vertex[x=-1,y=-7,L=2]{B112112}
  \Vertex[x=1,y=-7,L=1]{B112121}
  \Vertex[x=3,y=-7,L=2]{B112122}
  \Vertex[x=7,y=-7,L=1]{B121211}
  \Vertex[x=9,y=-7,L=2]{B121212}
  \SetVertexNormal[Shape = regular polygon,
  					LineWidth=1pt]
  \tikzset{VertexStyle/.append style =
{regular polygon sides=3, inner sep=0pt}}
  \Vertex[x=-10,y=-2,L=1]{A1}
  \Vertex[x=-2,y=-3,L=1]{B11}
  \Vertex[x=9,y=-5,L=2]{B1212}
  \SetVertexNormal[Shape = rectangle,
  					LineWidth = 1pt]
  \Vertex[x=-6,y=-4,L=1]{B111}
  \Vertex[x=5,y=-5,L=1]{B1211}
  \Vertex[x=-2,y=-6,L=1]{B11211}
  \Vertex[x=8,y=-6,L=1]{B12121}
  \SetVertexNormal[Shape = star,
  					LineWidth = 1pt]
  \tikzset{VertexStyle/.append style =
{star points=9,star point ratio=0.6, inner sep =0pt, minimum size=14pt}}
  \Vertex[x=0,y=-8,L=1]{B1121211}
  \Vertex[x=2,y=-8,L=2]{B1121212}  
  \SetVertexNoLabel
  \SetVertexNormal[Shape      = circle,
                   LineWidth  = 1pt]
  \Vertex{r}
  \Edge(A)(r)
  \Edge(A)(A1)
  \Edge(A)(A2)
  \Edge(A1)(A11)
  \Edge(A1)(A12)
  \Edge(A2)(A21)
  \Edge(A2)(A22)
  \Edge(B)(r)
  \Edge(B)(B1)
  \Edge(B)(B2)
  \Edge(B1)(B11)
  \Edge(B1)(B12)
  \Edge(B2)(B21)
  \Edge(B2)(B22)
  \Edge(B11)(B111)
  \Edge(B11)(B112)
  \Edge(B12)(B121)
  \Edge(B12)(B122)
  \Edge(B111)(B1111)
  \Edge(B111)(B1112)
  \Edge(B112)(B1121)
  \Edge(B112)(B1122)
  \Edge(B121)(B1211)
  \Edge(B121)(B1212)
  \Edge(B1111)(B11111)
  \Edge(B1111)(B11112)
  \Edge(B1112)(B11121)
  \Edge(B1112)(B11122)
  \Edge(B1121)(B11211)
  \Edge(B1121)(B11212)
  \Edge(B1211)(B12111)
  \Edge(B1211)(B12112)
  \Edge(B1212)(B12121)
  \Edge(B1212)(B12122)
  \Edge(B11121)(B111211)
  \Edge(B11121)(B111212)
  \Edge(B11211)(B112111)
  \Edge(B11211)(B112112)
  \Edge(B11212)(B112121)
  \Edge(B11212)(B112122)
  \Edge(B12121)(B121211)
  \Edge(B12121)(B121212)
  \Edge(B112121)(B1121211)
  \Edge(B112121)(B1121212)
\end{tikzpicture}}
\end{center}
\caption{The case tree. Each vertex represents the case whose label is given by appending the labels of vertices along the path from the root to it. The two special end-cases are in stars. The marked split-cases are in triangles if \textbf{P1} lost $1$ edge and in squares if he lost $2$.}
\label{fig:tree}
\end{figure}
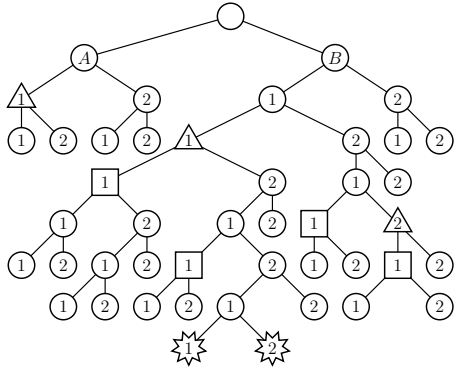

In this section we present a drawing strategy for \textbf{P2} in $\mathcal{R}(K_\omega\sqcup K_\omega,G)$. We construct it by considering various cases, according to the moves of \textbf{P1}. The following trick allows us to reduce significantly the number of cases considered: instead of considering all possible moves of \textbf{P1} up to isomorphism we often give him \emph{additional edges}, i.e. edges taken by \textbf{P1} which are not specified immediately after \textbf{P1}'s turn, but are obviously taken into account. All figures depict (part of) the configuration in $K^2$ only (the first edge of \textbf{P1} being in $K^1$), not featuring the additional edges. At a certain stage of the game, the current number of additional edges is marked as ``$+n$'' on the corresponding figure and the rest of \textbf{P1}'s edges are specified. The first edge of \textbf{P1} is considered specified, although it does not appear drawn. 
In all the figures, \textbf{P1}'s edges are drawn as dashed lines while \textbf{P2}'s edges are drawn as continuous lines. Whenever two figures are drawn for the same case, the left one corresponds to the initial state of the game before \textbf{P2}'s move under the assumption defining the present case, while the right one corresponds to the final state of the game before \textbf{P2}'s subsequent move. If only one figure is drawn, it corresponds to the final state of the game, as the initial state of the game is the same as the parent case with an additional assumption defining the present case.

The cases considered in the strategy naturally form a strict binary tree (see Figure~\ref{fig:tree}), whose leaves we call  \emph{end-cases} and whose internal nodes we call \emph{split-cases}. When possible, we will respect the convention that at each split-case we consider a particular edge \textbf{P2} is interested in, the left child (labeled by $1$) corresponds to the edge being already taken by \textbf{P1}, while the right child (labeled by $2$) corresponds to the edge being available for \textbf{P2} to take. 

Finally, if, at a certain stage of the game, \textbf{P1} has $k+1$ edges in total (including his first edge in $K^1$ and at most $k$ other edges in $K^2$), we say that \emph{\textbf{P1} lost $l$ edges}, if $e_\textbf{P1}(G)\le k-l$ for any $G$ in $K^2$. It is clear that the number of edges lost by \textbf{P2} is non-decreasing in the course of the game. For that reason we shall indicate the loss of edges as soon as it arises. More precisely, a split-case is \emph{marked} if at the final stage of the case \textbf{P1} is certain to have lost more edges than in the parent case. These cases are denoted by triangles (resp. squares) in Figure~\ref{fig:tree} if \textbf{P1} lost $1$ (resp. $2$) edges.

For the reader's convenience, in Section~\ref{subsec:case:1} we make a detailed analysis of \textbf{Case 1.} This analysis contains on one hand the strategy of \textbf{P2} in \textbf{Case 1.} and on the other hand the detailed verification of the hypotheses of Lemma~\ref{lem:2} in each of the end-cases. This allows us to conclude that \textbf{P2} has a drawing strategy in \textbf{Case 1.} In Section \ref{subsec:conditions} we present the automatic procedure to check that in all end-cases of \textbf{Case 2.}, with the exception of the two special end-cases, the hypotheses of Lemma~\ref{lem:2} are satisfied.  In Section~\ref{subsec:case:2} we give the strategy of \textbf{P2} in \textbf{Case 2.}, leaving the mechanical checks explained in Section~\ref{subsec:conditions} to the reader. Finally, the two \emph{special end-cases}, namely \textbf{\breakdot{Case 2.1.1.2.1.2.1.1.}} and \textbf{\breakdot{Case 2.1.1.2.1.2.1.2.}}, are treated separately in Section~\ref{s4.4}, as they do not allow a direct application of Lemma~\ref{lem:2}. Altogether, this allows us to conclude that \textbf{P2} has a drawing strategy.

\subsection{Split-case 1.}
\label{subsec:case:1}

\begin{figure}[H]
\begin{floatrow}[2]
\ffigbox
{\scalebox{.55}{
\begin{tikzpicture}
  \SetGraphUnit{3}
\GraphInit[vstyle=Normal]
  \SetVertexNormal[Shape      = circle,
                   LineWidth  = 1pt]
  \SetVertexMath
  \Vertex{A}
  \Vertex[x=1,y=1]{B}
  
  \SetVertexNormal[Shape = rectangle,
  					LineWidth = 1pt]
  \Vertex[x=0,y=2]{+1}
  \Edge(A)(B)
\end{tikzpicture}}}{
\caption{\textbf{General split-case.} Without loss of generality assume that the first edge \textbf{P1} takes is $XY$ in $K^1$. In response, \textbf{P2} takes the edge $AB$ in $K^2$. We mention that, excluding $XY$, \textbf{P1} has $1$ unspecified additional edge (not appearing in the figure) by ``+1''.}
\label{fig:gen}}

\ffigbox{
\centering
\scalebox{.55}{
\begin{tikzpicture}
  \SetGraphUnit{3}
\GraphInit[vstyle=Normal]
  \SetVertexNormal[Shape      = circle,
                   LineWidth  = 1pt]
  \SetVertexMath
  \Vertex{A}
  \Vertex[x=1,y=1]{B}
  
  \SetVertexNormal[Shape = rectangle,
  					LineWidth = 1pt]
  \Vertex[x=0,y=2]{+1^*}
  \Edge(A)(B)
\end{tikzpicture}
}
\scalebox{.55}{
\begin{tikzpicture}
  \SetGraphUnit{3}
\GraphInit[vstyle=Normal]
  \SetVertexNormal[Shape      = circle,
                   LineWidth  = 1pt]
  \SetVertexMath
  \Vertex{A}
  \Vertex[x=1,y=1]{B}
  \Vertex[x=2,y=0]{C}
  
  \SetVertexNormal[Shape = rectangle,
  					LineWidth = 1pt]
  \Vertex[x=0,y=2]{+2^*}
  \Edge(A)(B)
  \Edge(B)(C)
\end{tikzpicture}
}}{\caption{\textbf{Split-case 1.} Assume that the second edge taken by \textbf{P1} is either in $K^1$ or in $K^2$ and is incident with either $A$ or $B$. In the latter case, we may assume that it is incident with $B$. We indicate that either one of these is the case by a ``*'' in the figures. Then \textbf{P2} takes the edge $BC$, where $C$ is any free vertex in $K^2$.}
\label{fig:1}}
\end{floatrow}
\end{figure}

\paragraph{Marked split-case 1.1}
\begin{figure}[H]
\begin{floatrow}[3]
\ffigbox{
\centering
\scalebox{.55}{
\begin{tikzpicture}
  \SetGraphUnit{3}
\GraphInit[vstyle=Normal]
  \SetVertexNormal[Shape      = circle,
                   LineWidth  = 1pt]
  \SetVertexMath
  \Vertex{A}
  \Vertex[x=1,y=1]{B}
  \Vertex[x=2,y=0]{C}
  
  \SetVertexNormal[Shape = rectangle,
  					LineWidth = 1pt]
  \Vertex[x=0,y=2]{+1^*}
  \Edge(A)(B)
  \Edge(B)(C)
  \SetUpEdge[style={dashed}]
  \Edge(A)(C)
\end{tikzpicture}
}
\scalebox{.55}{
\begin{tikzpicture}
  \SetGraphUnit{3}
\GraphInit[vstyle=Normal]
  \SetVertexNormal[Shape      = circle,
                   LineWidth  = 1pt]
  \SetVertexMath
  \Vertex{A}
  \Vertex[x=1,y=1]{B}
  \Vertex[x=2,y=0]{C}
  \Vertex[x=1,y=2.5]{D}
  
  \SetVertexNormal[Shape = rectangle,
  					LineWidth = 1pt]
  \Vertex[x=0,y=2]{+3^*}
  \Edge(A)(B)
  \Edge(B)(C)
  \Edge(B)(D)
  \Edge(A)(D)
  \SetUpEdge[style={dashed}]
  \Edge(A)(C)
\end{tikzpicture}
}}{\caption{\textbf{Marked split-case \breakdot{1.1.}} Assume that \textbf{P1} has $AC$. Then \textbf{P2} takes $BD$, where $D$ is any free vertex in $K^2$. Note that at least one of the edges $DA$, $DC$ is not taken by any player, so assume without loss of generality that $DA$ is not taken. Then \textbf{P2} takes $DA$.\\\textbf{P1} lost an edge.}
\label{fig:1.1}}

\ffigbox{
\scalebox{.55}{
\begin{tikzpicture}
  \SetGraphUnit{3}
\GraphInit[vstyle=Normal]
  \SetVertexNormal[Shape      = circle,
                   LineWidth  = 1pt]
  \SetVertexMath
  \Vertex{A}
  \Vertex[x=1,y=1]{B}
  \Vertex[x=2,y=0]{C}
  \Vertex[x=1,y=2.5]{D}
  
  \SetVertexNormal[Shape = rectangle,
  					LineWidth = 1pt]
  \Vertex[x=0,y=2]{+2^*}
  \Edge(A)(B)
  \Edge(B)(C)
  \Edge(B)(D)
  \Edge(D)(A)
  \SetUpEdge[style={dashed}]
  \Edge(A)(C)
  \Edge(C)(D)
\end{tikzpicture}
}
\scalebox{.55}{
\begin{tikzpicture}
  \SetGraphUnit{3}
\GraphInit[vstyle=Normal]
  \SetVertexNormal[Shape      = circle,
                   LineWidth  = 1pt]
  \SetVertexMath
  \Vertex{A}
  \Vertex[x=1,y=1]{B}
  \Vertex[x=2,y=0]{C}
  \Vertex[x=1,y=2.5]{D}
  \Vertex[x=2,y=2]{E}
  
  \SetVertexNormal[Shape = rectangle,
  					LineWidth = 1pt]
  \Vertex[x=0,y=2]{+4^*}
  \Edge(A)(B)
  \Edge(B)(C)
  \Edge(D)(A)
  \Edge(B)(E)
  \Edge(D)(E)
  \SetUpEdge[lw=3pt]
  \Edge(B)(D)
  \SetUpEdge[style={dashed}]
  \Edge(A)(C)
  \Edge(C)(D)
\end{tikzpicture}
}}{\caption{\textbf{End-case \breakdot{1.1.1.}} Assume that \textbf{P1} has $DC$. Then \textbf{P2} takes $BE$, where $E$ is any free vertex in $K^2$. Note that at least one of the edges $ED$, $EA$ is not taken by any player, so assume without loss of generality that $ED$ is not taken. Then \textbf{P2} takes $ED$.\\The potential base is $BD$.}
\label{fig:1.1.1}}

\ffigbox{\scalebox{.55}{
\begin{tikzpicture}
  \SetGraphUnit{3}
\GraphInit[vstyle=Normal]
  \SetVertexNormal[Shape      = circle,
                   LineWidth  = 1pt]
  \SetVertexMath
  \Vertex{A}
  \Vertex[x=1,y=1]{B}
  \Vertex[x=2,y=0]{C}
  \Vertex[x=1,y=2.5]{D}
  
  \SetVertexNormal[Shape = rectangle,
  					LineWidth = 1pt]
  \Vertex[x=0,y=2]{+4^*}
  \Edge(A)(B)
  \Edge(B)(C)
  \Edge(D)(A)
  \Edge(C)(D)
  \SetUpEdge[lw=3pt]
  \Edge(B)(D)
  \SetUpEdge[style={dashed}]
  \Edge(A)(C)
\end{tikzpicture}
}}{\caption{\textbf{End-case \breakdot{1.1.2.}} Assume that \textbf{P1} does not have $DC$. Then \textbf{P2} takes this edge.\\The potential base is $BD$.}
\label{fig:1.1.2}
}
\end{floatrow}
\end{figure}

Let us begin by checking that in \textbf{Marked split-case 1.1.}, Figure~\ref{fig:1.1}, \textbf{P1} lost an edge. At this stage \textbf{P1} has a total of $5$ edges including the first edge $XY$, which is in $K^1$. We need to check that $e_{\textbf{P1}}(G)\le 3$ for any $G$ in $K^2$. Assume for the sake of contradiction that there exist a $G$ in $K^2$ such that $e_{\textbf{P1}}(G)\ge 4$. Then $G$ contains all edges of \textbf{P1} except $XY$ and contains no edges of \textbf{P2}. This forces the second edge of \textbf{P1} to be in $K^2$, and moreover, condition ``*'' forces this edge to be incident to $B$, say $BF$. In order for $G$ to contain both $AC$ and $BF$, the base $X_0X_1$ of $G$ needs to satisfy without loss of generality $X_0\in\{A,C\}$, as any edge in $G$ is incident with the base. Then $G$ contains the edge $X_0B$, as every vertex of $G$ is connected to $X_0$, so $X_0B$ needs not to be taken by \textbf{P2}. This yields the desired contradiction (see Figure~\ref{fig:1.1}).

In order to apply Lemma~\ref{lem:2} to \textbf{End-case 1.1.1.} and \textbf{End-case 1.1.2.} and conclude that in these cases \textbf{P2} has a drawing strategy, we are now ready to check that the hypotheses of Lemma~\ref{lem:2} are satisfied in these cases.

Firstly, let us check that condition \ref{cond:b} of Lemma~\ref{lem:2} holds. Since \textbf{P1} lost an edge in \textbf{Marked split-case 1.1.}, the same holds for \textbf{End-case 1.1.1.} and \textbf{End-case 1.1.2.}. Note that the total number of edges of \textbf{P1} including the first edge $XY$ which is in $K^1$ is $7$ and $6$, respectively, for the two end-cases. Therefore, in \textbf{End-case 1.1.1.} and \textbf{End-case 1.1.2.} each $G$ in $K^2$ satisfies $e_{\textbf{P1}}(G)\le 5$ and  $e_{\textbf{P1}}(G)\le 4$, respectively, as desired.

Secondly, let us check that condition \ref{cond:c} of Lemma~\ref{lem:2} holds. This follows immediately from the fact that in \textbf{End-case 1.1.1.} and \textbf{End-case 1.1.2.} \textbf{P1} has $4$ additional edges, so at most $5$ edges in $K^1$.

Thirdly, let us check that condition \ref{cond:a} of Lemma~\ref{lem:2} holds. We do so by applying Lemma~\ref{lem:3} to the thickened edge $BD$. Thus, we are left with checking that the hypotheses of Lemma~\ref{lem:3} are satisfied.

Condition \ref{cond:base} of Lemma~\ref{lem:3} is checked directly on the figures. Note that by \eqref{rem35} we have that $AC$ is good for $BD$, as both $AB$ and $AD$ are taken by \textbf{P2}. Recall that in \textbf{End-case 1.1.1.} and in \textbf{End-case 1.1.2.}  \textbf{P1} has at most 6 and 5 edges in $K^2$, respectively. Therefore, in \textbf{End-case 1.1.1.} and in \textbf{End-case 1.1.2.}  \textbf{P1} has at most 5 and 4 bad edges, respectively. It follows that condition \ref{cond:5bad} of Lemma~\ref{lem:3} holds. Furthermore, in \textbf{End-case 1.1.2.} condition \ref{cond:2d} of Lemma~\ref{lem:3} also holds. Finally, we are left with checking condition \ref{cond:2d} in \textbf{End-case 1.1.1.} i.e.\ that $CD$, together with the other 4 possible bad edges, cannot form a $2\Delta$-configuration. Indeed, $CD$ cannot appear in a $2\Delta$-configuration for $BD$, since $BC$ is taken by \textbf{P2}. This concludes the proof.

\begin{figure}[H]
\begin{floatrow}[3]

\ffigbox{\scalebox{.55}{
\begin{tikzpicture}
  \SetGraphUnit{3}
\GraphInit[vstyle=Normal]
  \SetVertexNormal[Shape      = circle,
                   LineWidth  = 1pt]
  \SetVertexMath
  \Vertex{A}
  \Vertex[x=1,y=1]{B}
  \Vertex[x=2,y=0]{C}
  
  \SetVertexNormal[Shape = rectangle,
  					LineWidth = 1pt]
  \Vertex[x=0,y=2]{+3^*}
  \Edge(A)(B)
  \Edge(B)(C)
  \Edge(A)(C)
\end{tikzpicture}
}
\scalebox{.55}{
\begin{tikzpicture}
  \SetGraphUnit{3}
\GraphInit[vstyle=Normal]
  \SetVertexNormal[Shape      = circle,
                   LineWidth  = 1pt]
  \SetVertexMath
  \Vertex{A}
  \Vertex[x=1,y=1]{B}
  \Vertex[x=2,y=0]{C}
  \Vertex[x=1,y=2.5]{D}
  
  \SetVertexNormal[Shape = rectangle,
  					LineWidth = 1pt]
  \Vertex[x=0,y=2]{+4^*}
  \Edge(A)(B)
  \Edge(B)(C)
  \Edge(C)(D)
  \Edge(A)(C)
\end{tikzpicture}
}}{\caption{\textbf{Split-case \breakdot{1.2.}} Assume that \textbf{P1} does not have $AC$. Then \textbf{P2} takes this edge and then plays the edge $CD$, where $D$ is a free vertex.}
\label{fig:1.2}}

\ffigbox{
\scalebox{.55}{
\begin{tikzpicture}
  \SetGraphUnit{3}
\GraphInit[vstyle=Normal]
  \SetVertexNormal[Shape      = circle,
                   LineWidth  = 1pt]
  \SetVertexMath
  \Vertex{A}
  \Vertex[x=1,y=1]{B}
  \Vertex[x=2,y=0]{C}
  \Vertex[x=1,y=2.5]{D}
  
  \SetVertexNormal[Shape = rectangle,
  					LineWidth = 1pt]
  \Vertex[x=0,y=2]{+3^*}
  \Edge(A)(B)
  \Edge(B)(C)
  \Edge(A)(C)
  \Edge(C)(D)
  \SetUpEdge[style={dashed}]
  \Edge(A)(D)
\end{tikzpicture}
}
\scalebox{.55}{
\begin{tikzpicture}
  \SetGraphUnit{3}
\GraphInit[vstyle=Normal]
  \SetVertexNormal[Shape      = circle,
                   LineWidth  = 1pt]
  \SetVertexMath
  \Vertex{A}
  \Vertex[x=1,y=1]{B}
  \Vertex[x=2,y=0]{C}
  \Vertex[x=1,y=2.5]{D}
  
  \SetVertexNormal[Shape = rectangle,
  					LineWidth = 1pt]
  \Vertex[x=0,y=2]{+4^*}
  \Edge(A)(B)
  \Edge(A)(C)
  \Edge(C)(D)
  \Edge(B)(D)
  \SetUpEdge[lw=3pt]
  \Edge(B)(C)
  \SetUpEdge[style={dashed}]
  \Edge(A)(D)
\end{tikzpicture}
}}{\caption{\textbf{End-case \breakdot{1.2.1.}} Assume that \textbf{P1} has $AD$. Then he does not have $BD$, as $D$ was a \textbf{P1}-free vertex before his move. Then \textbf{P2} takes $BD$.\\The potential base is $BC$.}
\label{fig:1.2.1}}

\ffigbox{\scalebox{.55}{
\begin{tikzpicture}
  \SetGraphUnit{3}
\GraphInit[vstyle=Normal]
  \SetVertexNormal[Shape      = circle,
                   LineWidth  = 1pt]
  \SetVertexMath
  \Vertex{A}
  \Vertex[x=1,y=1]{B}
  \Vertex[x=2,y=0]{C}
  \Vertex[x=1,y=2.5]{D}
  
  \SetVertexNormal[Shape = rectangle,
  					LineWidth = 1pt]
  \Vertex[x=0,y=2]{+5^*}
  \Edge(A)(B)
  \Edge(B)(C)
  \Edge(C)(D)
  \Edge(A)(D)
  \SetUpEdge[lw=3pt]
  \Edge(A)(C)
\end{tikzpicture}
}}{\caption{\textbf{End-case \breakdot{1.2.2.}} Assume that \textbf{P2} does not have $AD$. Then \textbf{P2} takes this edge.\\The potential base is $AC$.}
\label{fig:1.2.2}}
\end{floatrow}
\end{figure}

\paragraph{Split-case 1.2}
Note that \textbf{Split-case 1.2.} is not marked. In order to apply Lemma~\ref{lem:2} to \textbf{End-case 1.2.1.} and \textbf{End-case 1.2.2.} and conclude that in these cases \textbf{P2} has a drawing strategy, we shall check that the hypotheses of Lemma~\ref{lem:2} are satisfied.

In both \textbf{End-case 1.2.1.} and \textbf{End-case 1.2.2.} condition \ref{cond:b} of Lemma~\ref{lem:2} trivially holds as \textbf{P1} has at most $5$ edges in $K^2$. Condition \ref{cond:c} of Lemma~\ref{lem:2} is also automatic, as \textbf{P1} has $4$ and $5$ additional edges, respectively. 

In order to check that  condition \ref{cond:a} of Lemma~\ref{lem:2} holds, we shall check that the hypotheses of Lemma~\ref{lem:3} are satisfied. Condition \ref{cond:base} of Lemma~\ref{lem:3} is checked directly on Figures~\ref{fig:1.2.1} and \ref{fig:1.2.2}. Note that in \textbf{End-case 1.2.1.} the edge $AD$ is good for $BC$, as $AB$ and $AC$ are both taken by \textbf{P2}. It follows that there are at most $4$ bad edges, implying conditions \ref{cond:5bad} and \ref{cond:2d} of Lemma~\ref{lem:3}. Also note that in \textbf{End-case 1.2.2.} the second edge of \textbf{P1} is either in $K^1$ or is in $K^2$ incident to $B$ by condition ``*''. In either case this edge is good for $AC$, as both $AB$ and $BC$ are taken by \textbf{P2}. This allows us to conclude as in the previous case.

\subsection{Applying Lemma~\ref{lem:2}}
\label{subsec:conditions}
Let us now explain how the conditions of Lemma~\ref{lem:2} are verified in the remaining end-cases with the exception of the special ones which we shall discuss further in Section~\ref{s4.4}. The procedure generalises the one in the example cases discussed in detail in Section~\ref{subsec:case:1}.

\paragraph{Condition~\ref{cond:c} of Lemma~\ref{lem:2}}
In each of the end-cases (including the special ones), we mechanically check that \textbf{P1} has at most $5$ additional edges, so he has at most $6$ edges in $K^1$.

\paragraph{Condition~\ref{cond:b} of Lemma~\ref{lem:2}}
It is just a little harder to check that in all non-special end-cases for all $G$ in $K^2$, $e_{\textbf{P1}}(G) \leq 5$. For each  marked split-case (see Figure~\ref{fig:tree}) we see that \textbf{P1} loses (at least) $l$ edges for a certain number $l \in \{1, 2 \}$ which depends on the case. If \textbf{P1} loses $l$ edges in a given split-case, it follows that he also loses $l$ edges in all of its descendant cases. For each non-special end-case we refer to the closest ancestor marked split-case (if any) to see how many edges \textbf{P1} lost (see Figure~\ref{fig:tree}) and check that condition~\ref{cond:b} holds.

In order to establish that \textbf{P1} loses $l$ edges in a marked split-case, we proceed as follows. Let $k+1$ be the total number of edges taken by \textbf{P1}. We consider all non-isomorphic edges $X_0 X_1$ in $K^2$, such that $X_0X_1$ is not taken by \textbf{P2} (so $X_0X_1$ could potentially become the base of a $G$ that \textbf{P1} constructs). We count the number of edges in $K^2$ taken by \textbf{P1} which are of the form $X_iY$ with $YX_{1-i}$ not taken by \textbf{P2}, and add the edge $X_0X_1$ if it is taken by \textbf{P1}. This number bounds $e_{\textbf{P1}}(G)$ for copies of $G$ with base $X_0X_1$. We check that for each choice of $X_0X_1$ there are at most $k-l$ such edges. To exclude a large number of edges $X_0X_1$ from the very beginning, we first investigate $\deg_{\textbf{P1}}(v)$ for all vertices $v$. 

\paragraph{Condition~\ref{cond:a} of Lemma~\ref{lem:2}}
In each of the non-special end-cases, we mechanically check that \textbf{P2} has a potential base $A_0A_1$ -- which is declared and marked in all pictures by a thickened edge. To do so it is enough to check the conditions of Lemma~\ref{lem:3}. Firstly, we inspect that \textbf{P1} has at most $5$ bad edges for $A_0A_1$ (condition \ref{cond:5bad} of Lemma~\ref{lem:3}), using \eqref{rem35} to establish that some edges are good for $A_0A_1$ and further that the bad edges do not form a $2\Delta$-configuration if there are $5$ of them (condition \ref{cond:2d} of Lemma~\ref{lem:3}). Finally, we inspect that there exist two vertices $B_1, B_2$ such that \textbf{P2} has the edges $A_0B_1, A_0B_2, A_1B_1, A_1B_2$ (condition \ref{cond:base} of Lemma~\ref{lem:3}), so we can conclude.

\begin{rem}While the verification of the conditions of Lemma~\ref{lem:2} is very easy given the strategy for \textbf{P2}, it is a central part of the paper to determine that right strategy. The idea of using additional edges and reducing drastically the number of cases helped us achieve this goal.
\end{rem}

\subsection{Split-case 2.}
\label{subsec:case:2}
We now describe the rest of the strategy. Recall the \textbf{General split-case}, Figure~\ref{fig:gen}.

\begin{figure}[H]
\begin{floatrow}[3]
\ffigbox{\scalebox{.55}{
\begin{tikzpicture}
  \SetGraphUnit{3}
\GraphInit[vstyle=Normal]
  \SetVertexNormal[Shape      = circle,
                   LineWidth  = 1pt]
  \SetVertexMath
  \Vertex{A}
  \Vertex[x=1,y=1]{B}
  \Vertex[x=0,y=-2]{C}
  \Vertex[x=2,y=-2]{D}
  
  \SetVertexNormal[Shape = rectangle,
  					LineWidth = 1pt]
  \Edge(A)(B)
  \SetUpEdge[style={dashed}]
  \Edge(C)(D)
\end{tikzpicture}
}
\scalebox{.55}{
\begin{tikzpicture}
  \SetGraphUnit{3}
\GraphInit[vstyle=Normal]
  \SetVertexNormal[Shape      = circle,
                   LineWidth  = 1pt]
  \SetVertexMath
  \Vertex{A}
  \Vertex[x=1,y=1]{B}
  \Vertex[x=0,y=-2]{C}
  \Vertex[x=2,y=-2]{D}
  \Vertex[x=2,y=0]{E}
  
  \SetVertexNormal[Shape = rectangle,
  					LineWidth = 1pt]
  \Vertex[x=0,y=2]{+1}
  \Edge(A)(B)
  \Edge(B)(E)
  \SetUpEdge[style={dashed}]
  \Edge(C)(D)
\end{tikzpicture}
}}{\caption{\textbf{Split-case 2.} Assume that the second edge taken by \textbf{P1} is $CD$ with $C$ and $D$ free vertices in $K^2$. Then \textbf{P2} takes $BE$, where $E$ is a free vertex in $K^2$.}
\label{fig:2}}

\ffigbox{
\scalebox{.55}{
\begin{tikzpicture}
  \SetGraphUnit{3}
\GraphInit[vstyle=Normal]
  \SetVertexNormal[Shape      = circle,
                   LineWidth  = 1pt]
  \SetVertexMath
  \Vertex{A}
  \Vertex[x=1,y=1]{B}
  \Vertex[x=0,y=-2]{C}
  \Vertex[x=2,y=-2]{D}
  \Vertex[x=2,y=0]{E}
  
  \SetVertexNormal[Shape = rectangle,
  					LineWidth = 1pt]
  \Edge(A)(B)
  \Edge(B)(E)
  \SetUpEdge[style={dashed}]
  \Edge(C)(D)
  \Edge(A)(E)
\end{tikzpicture}
}
\scalebox{.55}{
\begin{tikzpicture}
  \SetGraphUnit{3}
\GraphInit[vstyle=Normal]
  \SetVertexNormal[Shape      = circle,
                   LineWidth  = 1pt]
  \SetVertexMath
  \Vertex{A}
  \Vertex[x=1,y=1]{B}
  \Vertex[x=0,y=-2]{C}
  \Vertex[x=2,y=-2]{D}
  \Vertex[x=2,y=0]{E}
  
  \SetVertexNormal[Shape = rectangle,
  					LineWidth = 1pt]
  \Vertex[x=0,y=2]{+1}
  \Edge(A)(B)
  \Edge(B)(E)
  \Edge(C)(E)
  \SetUpEdge[style={dashed}]
  \Edge(C)(D)
  \Edge(A)(E)
\end{tikzpicture}
}}{\caption{\textbf{Split-case \breakdot{2.1.}} Assume that \textbf{P1} has $AE$. Then \textbf{P2} takes $CE$.
}
\label{fig:2.1}}

\ffigbox{\scalebox{.55}{
\begin{tikzpicture}
  \SetGraphUnit{3}
\GraphInit[vstyle=Normal]
  \SetVertexNormal[Shape      = circle,
                   LineWidth  = 1pt]
  \SetVertexMath
  \Vertex{A}
  \Vertex[x=1,y=1]{B}
  \Vertex[x=0,y=-2]{C}
  \Vertex[x=2,y=-2]{D}
  \Vertex[x=2,y=0]{E}
  
  \SetVertexNormal[Shape = rectangle,
  					LineWidth = 1pt]
  \Edge(A)(B)
  \Edge(B)(E)
  \Edge(E)(C)
  \SetUpEdge[style={dashed}]
  \Edge(C)(D)
  \Edge(A)(E)
  \Edge(B)(C)
\end{tikzpicture}
}
\scalebox{.55}{
\begin{tikzpicture}
  \SetGraphUnit{3}
\GraphInit[vstyle=Normal]
  \SetVertexNormal[Shape      = circle,
                   LineWidth  = 1pt]
  \SetVertexMath
  \Vertex{A}
  \Vertex[x=1,y=1]{B}
  \Vertex[x=0,y=-2]{C}
  \Vertex[x=2,y=-2]{D}
  \Vertex[x=2,y=0]{E}
  \Vertex[x=1,y=2.5]{F}
  
  \SetVertexNormal[Shape = rectangle,
  					LineWidth = 1pt]
  \Vertex[x=0,y=2]{+1}
  \Edge(A)(B)
  \Edge(B)(E)
  \Edge(E)(C)
  \Edge(B)(F)
  \SetUpEdge[style={dashed}]
  \Edge(C)(D)
  \Edge(A)(E)
  \Edge(B)(C)
\end{tikzpicture}
}}{\caption{\textbf{Marked split-case \breakdot{2.1.1.}} Assume that \textbf{P1} has $BC$. Then \textbf{P2} takes $BF$, where $F$ is a free vertex in $K^2$.\\\textbf{P1} lost an edge.}
\label{fig:2.1.1}}
\end{floatrow}
\end{figure}

\begin{figure}[H]
\begin{floatrow}[3]
\ffigbox{{
\scalebox{.55}{
\begin{tikzpicture}
  \SetGraphUnit{3}
\GraphInit[vstyle=Normal]
  \SetVertexNormal[Shape      = circle,
                   LineWidth  = 1pt]
  \SetVertexMath
  \Vertex{A}
  \Vertex[x=1,y=1]{B}
  \Vertex[x=0,y=-2]{C}
  \Vertex[x=2,y=-2]{D}
  \Vertex[x=2,y=0]{E}
  \Vertex[x=1,y=2.5]{F}
  
  \SetVertexNormal[Shape = rectangle,
  					LineWidth = 1pt]
  \Edge(A)(B)
  \Edge(B)(E)
  \Edge(E)(C)
  \Edge(B)(F)
  \SetUpEdge[style={dashed}]
  \Edge(C)(D)
  \Edge(A)(E)
  \Edge(B)(C)
  \Edge(E)(F)
\end{tikzpicture}
}
\scalebox{.55}{
\begin{tikzpicture}
  \SetGraphUnit{3}
\GraphInit[vstyle=Normal]
  \SetVertexNormal[Shape      = circle,
                   LineWidth  = 1pt]
  \SetVertexMath
  \Vertex{A}
  \Vertex[x=1,y=1]{B}
  \Vertex[x=0,y=-2]{C}
  \Vertex[x=2,y=-2]{D}
  \Vertex[x=2,y=0]{E}
  \Vertex[x=1,y=2.5]{F}
  
  \SetVertexNormal[Shape = rectangle,
  					LineWidth = 1pt]
  \Vertex[x=0,y=2]{+1}
  \Edge(A)(B)
  \Edge(B)(E)
  \Edge(E)(C)
  \Edge(B)(F)
  \Edge(A)(F)
  \SetUpEdge[style={dashed}]
  \Edge(C)(D)
  \Edge(A)(E)
  \Edge(B)(C)
  \Edge(E)(F)
\end{tikzpicture}
}}}{\caption{\textbf{Marked split-case \breakdot{2.1.1.1.}} Assume that \textbf{P1} has $EF$. Then \textbf{P2} takes $AF$.\\\textbf{P1} lost two edges.}
\label{fig:2.1.1.1}}

\ffigbox{
\scalebox{.55}{
\begin{tikzpicture}
  \SetGraphUnit{3}
\GraphInit[vstyle=Normal]
  \SetVertexNormal[Shape      = circle,
                   LineWidth  = 1pt]
  \SetVertexMath
  \Vertex{A}
  \Vertex[x=1,y=1]{B}
  \Vertex[x=0,y=-2]{C}
  \Vertex[x=2,y=-2]{D}
  \Vertex[x=2,y=0]{E}
  \Vertex[x=1,y=2.5]{F}
  
  \SetVertexNormal[Shape = rectangle,
  					LineWidth = 1pt]
  \Edge(A)(B)
  \Edge(B)(E)
  \Edge(E)(C)
  \Edge(B)(F)
  \Edge(A)(F)
  \SetUpEdge[style={dashed}]
  \Edge(C)(D)
  \Edge(A)(E)
  \Edge(B)(C)
  \Edge(E)(F)
  \Edge(B)(D)
\end{tikzpicture}
}
\scalebox{.55}{
\begin{tikzpicture}
  \SetGraphUnit{3}
\GraphInit[vstyle=Normal]
  \SetVertexNormal[Shape      = circle,
                   LineWidth  = 1pt]
  \SetVertexMath
  \Vertex{A}
  \Vertex[x=1,y=1]{B}
  \Vertex[x=0,y=-2]{C}
  \Vertex[x=2,y=-2]{D}
  \Vertex[x=2,y=0]{E}
  \Vertex[x=1,y=2.5]{F}
  \Vertex[x=2,y=1.5]{I}
  
  \SetVertexNormal[Shape = rectangle,
  					LineWidth = 1pt]
  \Vertex[x=0,y=2]{+1}
  \Edge(A)(B)
  \Edge(B)(E)
  \Edge(E)(C)
  \Edge(B)(F)
  \Edge(A)(F)
  \Edge(F)(I)
  \SetUpEdge[style={dashed}]
  \Edge(C)(D)
  \Edge(A)(E)
  \Edge(B)(C)
  \Edge(B)(D)
  \Edge(E)(F)
\end{tikzpicture}
}}{\caption{\textbf{Split-case \breakdot{2.1.1.1.1.}} Assume that \textbf{P1} has $BD$. Then \textbf{P2} takes $FI$, where $I$ is a free vertex in $K^2$.}
\label{fig:2.1.1.1.1}}

\ffigbox{
\scalebox{.55}{
\begin{tikzpicture}
  \SetGraphUnit{3}
\GraphInit[vstyle=Normal]
  \SetVertexNormal[Shape      = circle,
                   LineWidth  = 1pt]
  \SetVertexMath
  \Vertex{A}
  \Vertex[x=1,y=1]{B}
  \Vertex[x=0,y=-2]{C}
  \Vertex[x=2,y=-2]{D}
  \Vertex[x=2,y=0]{E}
  \Vertex[x=1,y=2.5]{F}
  \Vertex[x=2,y=1.5]{I}
  
  \SetVertexNormal[Shape = rectangle,
  					LineWidth = 1pt]
  \Edge(A)(B)
  \Edge(B)(E)
  \Edge(E)(C)
  \Edge(B)(F)
  \Edge(A)(F)
  \Edge(F)(I)
  \SetUpEdge[style={dashed}]
  \Edge(C)(D)
  \Edge(A)(E)
  \Edge(B)(C)
  \Edge(E)(F)
  \Edge(B)(D)
  \SetUpEdge[style={dashed,bend right=30}]
  \Edge(A)(I)
\end{tikzpicture}
}
\scalebox{.55}{
\begin{tikzpicture}
  \SetGraphUnit{3}
\GraphInit[vstyle=Normal]
  \SetVertexNormal[Shape      = circle,
                   LineWidth  = 1pt]
  \SetVertexMath
  \Vertex{A}
  \Vertex[x=1,y=1]{B}
  \Vertex[x=0,y=-2]{C}
  \Vertex[x=2,y=-2]{D}
  \Vertex[x=2,y=0]{E}
  \Vertex[x=1,y=2.5]{F}
  \Vertex[x=2,y=1.5]{I}
  
  \SetVertexNormal[Shape = rectangle,
  					LineWidth = 1pt]
  \Vertex[x=0,y=2]{+1}
  \Edge(A)(B)
  \Edge(B)(E)
  \Edge(E)(C)
  \Edge(A)(F)
  \Edge(F)(I)
  \Edge(B)(I)
  \SetUpEdge[lw=3pt]
  \Edge(B)(F)
  \SetUpEdge[style={dashed}]
  \Edge(C)(D)
  \Edge(A)(E)
  \Edge(B)(C)
  \Edge(E)(F)
  \Edge(B)(D)
  \SetUpEdge[style={dashed,bend right=30}]
  \Edge(A)(I)
\end{tikzpicture}
}}{\caption{\textbf{End-case \breakdot{2.1.1.1.1.1.}} Assume that \textbf{P1} has $AI$. Then \textbf{P2} takes $BI$.\\
The potential base is $BF$.}
\label{fig:2.1.1.1.1.1}
}
\end{floatrow}
\end{figure}

\begin{figure}[H]
\begin{floatrow}[3]

\ffigbox{\scalebox{.55}{
\begin{tikzpicture}
  \SetGraphUnit{3}
\GraphInit[vstyle=Normal]
  \SetVertexNormal[Shape      = circle,
                   LineWidth  = 1pt]
  \SetVertexMath
  \Vertex{A}
  \Vertex[x=1,y=1]{B}
  \Vertex[x=0,y=-2]{C}
  \Vertex[x=2,y=-2]{D}
  \Vertex[x=2,y=0]{E}
  \Vertex[x=1,y=2.5]{F}
  \Vertex[x=2,y=1.5]{I}
  
  \SetVertexNormal[Shape = rectangle,
  					LineWidth = 1pt]
  \Vertex[x=0,y=2]{+2}
  \Edge(A)(B)
  \Edge(B)(E)
  \Edge(E)(C)
  \Edge(B)(F)
  \Edge(F)(I)
  \SetUpEdge[style={bend right=30}]
  \Edge(A)(I)
  \SetUpEdge[lw=3pt]
  \Edge(A)(F)
  \SetUpEdge[style={dashed}]
  \Edge(C)(D)
  \Edge(A)(E)
  \Edge(B)(C)
  \Edge(E)(F)
  \Edge(B)(D)
\end{tikzpicture}
}}{\caption{\textbf{End-case \breakdot{2.1.1.1.1.2.}} Assume that \textbf{P1} does not have $AI$. Then \textbf{P2} takes this edge.\\The potential base is $AF$.}
\label{fig:2.1.1.1.1.2}}

\ffigbox{{
\scalebox{.55}{
\begin{tikzpicture}
  \SetGraphUnit{3}
\GraphInit[vstyle=Normal]
  \SetVertexNormal[Shape      = circle,
                   LineWidth  = 1pt]
  \SetVertexMath
  \Vertex{A}
  \Vertex[x=1,y=1]{B}
  \Vertex[x=0,y=-2]{C}
  \Vertex[x=2,y=-2]{D}
  \Vertex[x=2,y=0]{E}
  \Vertex[x=1,y=2.5]{F}
  
  \SetVertexNormal[Shape = rectangle,
  					LineWidth = 1pt]
  \Vertex[x=0,y=2]{+1}
  \Edge(A)(B)
  \Edge(B)(E)
  \Edge(E)(C)
  \Edge(B)(F)
  \Edge(A)(F)
  \SetUpEdge[style={dashed}]
  \Edge(C)(D)
  \Edge(A)(E)
  \Edge(B)(C)
  \Edge(E)(F)
\end{tikzpicture}
}}}{\caption{\textbf{Split-case \breakdot{2.1.1.1.2.}} Assume that \textbf{P1} does not have $BD$.}
\label{fig:2.1.1.1.2}}

\ffigbox{
\scalebox{.55}{
\begin{tikzpicture}
  \SetGraphUnit{3}
\GraphInit[vstyle=Normal]
  \SetVertexNormal[Shape      = circle,
                   LineWidth  = 1pt]
  \SetVertexMath
  \Vertex{A}
  \Vertex[x=1,y=1]{B}
  \Vertex[x=0,y=-2]{C}
  \Vertex[x=2,y=-2]{D}
  \Vertex[x=2,y=0]{E}
  \Vertex[x=1,y=2.5]{F}
  
  \SetVertexNormal[Shape = rectangle,
  					LineWidth = 1pt]
  \Edge(A)(B)
  \Edge(B)(E)
  \Edge(E)(C)
  \Edge(B)(F)
  \Edge(A)(F)
  \SetUpEdge[style={dashed}]
  \Edge(C)(D)
  \Edge(A)(E)
  \Edge(B)(C)
  \Edge(E)(F)
  \Edge(A)(D)
\end{tikzpicture}
}
\scalebox{.55}{
\begin{tikzpicture}
  \SetGraphUnit{3}
\GraphInit[vstyle=Normal]
  \SetVertexNormal[Shape      = circle,
                   LineWidth  = 1pt]
  \SetVertexMath
  \Vertex{A}
  \Vertex[x=1,y=1]{B}
  \Vertex[x=0,y=-2]{C}
  \Vertex[x=2,y=-2]{D}
  \Vertex[x=2,y=0]{E}
  \Vertex[x=1,y=2.5]{F}
  \Vertex[x=0,y=1.5]{I}
  
  \SetVertexNormal[Shape = rectangle,
  					LineWidth = 1pt]
  \Vertex[x=2,y=2]{+1}
  \Edge(A)(B)
  \Edge(B)(E)
  \Edge(E)(C)
  \Edge(B)(F)
  \Edge(F)(I)
  \Edge(A)(F)
  \SetUpEdge[style={dashed}]
  \Edge(C)(D)
  \Edge(A)(E)
  \Edge(B)(C)
  \Edge(E)(F)
  \Edge(A)(D)
\end{tikzpicture}
}}{\caption{\textbf{Split-case \breakdot{2.1.1.1.2.1.}} Assume that \textbf{P1} has $AD$ or $DF$ and without loss of generality let this edge be $AD$. Then \textbf{P2} takes $FI$, where $I$ is a free vertex in $K^2$.}
\label{fig:2.1.1.1.2.1}}
\end{floatrow}
\end{figure}

\begin{figure}[H]
\begin{floatrow}[3]
\ffigbox{\scalebox{.55}{
\begin{tikzpicture}
  \SetGraphUnit{3}
\GraphInit[vstyle=Normal]
  \SetVertexNormal[Shape      = circle,
                   LineWidth  = 1pt]
  \SetVertexMath
  \Vertex{A}
  \Vertex[x=1,y=1]{B}
  \Vertex[x=0,y=-2]{C}
  \Vertex[x=2,y=-2]{D}
  \Vertex[x=2,y=0]{E}
  \Vertex[x=1,y=2.5]{F}
  \Vertex[x=0,y=1.5]{I}
  
  \SetVertexNormal[Shape = rectangle,
  					LineWidth = 1pt]
  \Edge(A)(B)
  \Edge(B)(E)
  \Edge(E)(C)
  \Edge(B)(F)
  \Edge(A)(F)
  \Edge(F)(I)
  \SetUpEdge[style={dashed}]
  \Edge(C)(D)
  \Edge(A)(E)
  \Edge(B)(C)
  \Edge(E)(F)
  \Edge(A)(D)
  \Edge(B)(I)
\end{tikzpicture}
}
\scalebox{.55}{
\begin{tikzpicture}
  \SetGraphUnit{3}
\GraphInit[vstyle=Normal]
  \SetVertexNormal[Shape      = circle,
                   LineWidth  = 1pt]
  \SetVertexMath
  \Vertex{A}
  \Vertex[x=1,y=1]{B}
  \Vertex[x=0,y=-2]{C}
  \Vertex[x=2,y=-2]{D}
  \Vertex[x=2,y=0]{E}
  \Vertex[x=1,y=2.5]{F}
  \Vertex[x=0,y=1.5]{I}
  
  \SetVertexNormal[Shape = rectangle,
  					LineWidth = 1pt]
  \Vertex[x=2,y=2]{+1}
  \Edge(A)(B)
  \Edge(B)(E)
  \Edge(E)(C)
  \Edge(B)(F)
  \Edge(F)(I)
  \Edge(A)(I)
  \SetUpEdge[lw=3pt]
  \Edge(A)(F)
  \SetUpEdge[style={dashed}]
  \Edge(C)(D)
  \Edge(A)(E)
  \Edge(B)(C)
  \Edge(E)(F)
  \Edge(A)(D)
  \Edge(B)(I)
\end{tikzpicture}
}}{\caption{\textbf{End-case \breakdot{2.1.1.1.2.1.1.}} Assume that \textbf{P1} has $BI$. Then \textbf{P2} takes $AI$.\\The potential base is $AF$.}
\label{fig:2.1.1.1.2.1.1}}

\ffigbox{\scalebox{.55}{
\begin{tikzpicture}
  \SetGraphUnit{3}
\GraphInit[vstyle=Normal]
  \SetVertexNormal[Shape      = circle,
                   LineWidth  = 1pt]
  \SetVertexMath
  \Vertex{A}
  \Vertex[x=1,y=1]{B}
  \Vertex[x=0,y=-2]{C}
  \Vertex[x=2,y=-2]{D}
  \Vertex[x=2,y=0]{E}
  \Vertex[x=1,y=2.5]{F}
  \Vertex[x=0,y=1.5]{I}
  
  \SetVertexNormal[Shape = rectangle,
  					LineWidth = 1pt]
  \Vertex[x=2,y=2]{+2}
  \Edge(A)(B)
  \Edge(B)(E)
  \Edge(E)(C)
  \Edge(A)(F)
  \Edge(F)(I)
  \Edge(B)(I)
  \SetUpEdge[lw=3pt]
  \Edge(B)(F)
  \SetUpEdge[style={dashed}]
  \Edge(C)(D)
  \Edge(A)(E)
  \Edge(B)(C)
  \Edge(E)(F)
  \Edge(A)(D)
\end{tikzpicture}
}}{\caption{\textbf{End-case \breakdot{2.1.1.1.2.1.2.}} Assume that \textbf{P1} does not have $BI$. Then \textbf{P2} takes this edge.\\The potential base is $BF$.}
\label{fig:2.1.1.1.2.1.2}}

\ffigbox{\scalebox{.55}{
\begin{tikzpicture}
  \SetGraphUnit{3}
\GraphInit[vstyle=Normal]
  \SetVertexNormal[Shape      = circle,
                   LineWidth  = 1pt]
  \SetVertexMath
  \Vertex{A}
  \Vertex[x=1,y=1]{B}
  \Vertex[x=0,y=-2]{C}
  \Vertex[x=2,y=-2]{D}
  \Vertex[x=2,y=0]{E}
  \Vertex[x=1,y=2.5]{F}
  
  \SetVertexNormal[Shape = rectangle,
  					LineWidth = 1pt]
  \Vertex[x=0,y=2]{+3}
  \Edge(B)(F)
  \Edge(B)(E)
  \Edge(E)(C)
  \Edge(A)(F)
  \Edge(B)(D)
  \Edge(A)(D)
  \SetUpEdge[lw=3pt]
  \Edge(A)(B)
  \SetUpEdge[style={dashed}]
  \Edge(C)(D)
  \Edge(A)(E)
  \Edge(B)(C)
  \Edge(E)(F)
\end{tikzpicture}
}}{\caption{\textbf{End-case \breakdot{2.1.1.1.2.2.}} Assume that \textbf{P1} has none of $DA$, $DB$ or $DF$. Then \textbf{P1} takes $BD$ and then one of $AD$ and $DF$, which is not taken. Without loss of generality, let this edge be $AD$.\\The potential base is $AB$.}
\label{fig:2.1.1.1.2.2}}
\end{floatrow}
\end{figure}

\begin{figure}[H]
\begin{floatrow}[3]
\ffigbox{\scalebox{.55}{
\begin{tikzpicture}
  \SetGraphUnit{3}
\GraphInit[vstyle=Normal]
  \SetVertexNormal[Shape      = circle,
                   LineWidth  = 1pt]
  \SetVertexMath
  \Vertex{A}
  \Vertex[x=1,y=1]{B}
  \Vertex[x=0,y=-2]{C}
  \Vertex[x=2,y=-2]{D}
  \Vertex[x=2,y=0]{E}
  \Vertex[x=1,y=2.5]{F}
  
  \SetVertexNormal[Shape = rectangle,
  					LineWidth = 1pt]
  \Vertex[x=2,y=2]{+2}
  \Edge(A)(B)
  \Edge(B)(E)
  \Edge(E)(C)
  \Edge(B)(F)
  \Edge(E)(F)
  \SetUpEdge[style={dashed}]
  \Edge(C)(D)
  \Edge(A)(E)
  \Edge(B)(C)
\end{tikzpicture}
}}{\caption{\textbf{Split-case \breakdot{2.1.1.2.}} Assume that \textbf{P1} does not have $EF$. Then \textbf{P2} takes this edge.
}
\label{fig:2.1.1.2}}

\ffigbox{\scalebox{.55}{
\begin{tikzpicture}
  \SetGraphUnit{3}
\GraphInit[vstyle=Normal]
  \SetVertexNormal[Shape      = circle,
                   LineWidth  = 1pt]
  \SetVertexMath
  \Vertex{A}
  \Vertex[x=1,y=1]{B}
  \Vertex[x=0,y=-2]{C}
  \Vertex[x=2,y=-2]{D}
  \Vertex[x=2,y=0]{E}
  \Vertex[x=1,y=2.5]{F}
  
  \SetVertexNormal[Shape = rectangle,
  					LineWidth = 1pt]
  \Vertex[x=2,y=2]{+1}
  \Edge(A)(B)
  \Edge(B)(E)
  \Edge(E)(C)
  \Edge(B)(F)
  \Edge(E)(F)
  \SetUpEdge[style={dashed}]
  \Edge(C)(D)
  \Edge(A)(E)
  \Edge(B)(C)
  \Edge[style={dashed, bend left=40}](C)(F)
\end{tikzpicture}
}}{\caption{\textbf{Split-case \breakdot{2.1.1.2.1.}} Assume that \textbf{P1} has $CF$.}
\label{fig:2.1.1.2.1}}

\ffigbox{\scalebox{.55}{
\begin{tikzpicture}
  \SetGraphUnit{3}
\GraphInit[vstyle=Normal]
  \SetVertexNormal[Shape      = circle,
                   LineWidth  = 1pt]
  \SetVertexMath
  \Vertex{A}
  \Vertex[x=1,y=1]{B}
  \Vertex[x=0,y=-2]{C}
  \Vertex[x=2,y=-2]{D}
  \Vertex[x=2,y=0]{E}
  \Vertex[x=1,y=2.5]{F}
  
  \SetVertexNormal[Shape = rectangle,
  					LineWidth = 1pt]
  \Edge(A)(B)
  \Edge(B)(E)
  \Edge(E)(C)
  \Edge(B)(F)
  \Edge(E)(F)
  \SetUpEdge[style={dashed}]
  \Edge(C)(D)
  \Edge(A)(E)
  \Edge(B)(C)
  \Edge(A)(F)
  \Edge[style={dashed, bend left=40}](C)(F)
\end{tikzpicture}
}
\scalebox{.55}{
\begin{tikzpicture}
  \SetGraphUnit{3}
\GraphInit[vstyle=Normal]
  \SetVertexNormal[Shape      = circle,
                   LineWidth  = 1pt]
  \SetVertexMath
  \Vertex{A}
  \Vertex[x=1,y=1]{B}
  \Vertex[x=0,y=-2]{C}
  \Vertex[x=2,y=-2]{D}
  \Vertex[x=2,y=0]{E}
  \Vertex[x=1,y=2.5]{F}
  \Vertex[x=2,y=1.5]{I}
  
  \SetVertexNormal[Shape = rectangle,
  					LineWidth = 1pt]
  \Vertex[x=2,y=2.5]{+1}
  \Edge(A)(B)
  \Edge(B)(E)
  \Edge(E)(C)
  \Edge(B)(F)
  \Edge(E)(I)
  \Edge(E)(F)
  \SetUpEdge[style={dashed}]
  \Edge(C)(D)
  \Edge(A)(E)
  \Edge(B)(C)
  \Edge(A)(F)
  \Edge[style={dashed, bend left=40}](C)(F)
\end{tikzpicture}
}}{\caption{\textbf{Marked split-case \breakdot{2.1.1.2.1.1.}} Assume that \textbf{P1} has $AF$. Then \textbf{P2} takes $EI$, where $I$ is a free vertex in $K^2$.\\\textbf{P1} lost 2 edges.}
\label{fig:2.1.1.2.1.1}}
\end{floatrow}
\end{figure}

\begin{figure}[H]
\begin{floatrow}[3]
\ffigbox{\scalebox{.55}{
\begin{tikzpicture}
  \SetGraphUnit{3}
\GraphInit[vstyle=Normal]
  \SetVertexNormal[Shape      = circle,
                   LineWidth  = 1pt]
  \SetVertexMath
  \Vertex{A}
  \Vertex[x=1,y=1]{B}
  \Vertex[x=0,y=-2]{C}
  \Vertex[x=2,y=-2]{D}
  \Vertex[x=2,y=0]{E}
  \Vertex[x=1,y=2.5]{F}
  \Vertex[x=2,y=1.5]{I}
  
  \SetVertexNormal[Shape = rectangle,
  					LineWidth = 1pt]
  \Edge(A)(B)
  \Edge(B)(E)
  \Edge(E)(C)
  \Edge(B)(F)
  \Edge(E)(I)
  \Edge(E)(F)
  \SetUpEdge[style={dashed}]
  \Edge(C)(D)
  \Edge(A)(E)
  \Edge(B)(C)
  \Edge(A)(F)
  \Edge(B)(I)
  \Edge[style={dashed, bend left=40}](C)(F)
\end{tikzpicture}
}
\scalebox{.55}{
\begin{tikzpicture}
  \SetGraphUnit{3}
\GraphInit[vstyle=Normal]
  \SetVertexNormal[Shape      = circle,
                   LineWidth  = 1pt]
  \SetVertexMath
  \Vertex{A}
  \Vertex[x=1,y=1]{B}
  \Vertex[x=0,y=-2]{C}
  \Vertex[x=2,y=-2]{D}
  \Vertex[x=2,y=0]{E}
  \Vertex[x=1,y=2.5]{F}
  \Vertex[x=2,y=1.5]{I}
  
  \SetVertexNormal[Shape = rectangle,
  					LineWidth = 1pt]
  \Vertex[x=2,y=2.5]{+1}
  \Edge(A)(B)
  \Edge(B)(E)
  \Edge(E)(C)
  \Edge(B)(F)
  \Edge(E)(I)
  \Edge(F)(I)
  \SetUpEdge[lw=3pt]
  \Edge(E)(F)
  \SetUpEdge[style={dashed}]
  \Edge(C)(D)
  \Edge(A)(E)
  \Edge(B)(C)
  \Edge(A)(F)
  \Edge(B)(I)
  \Edge[style={dashed, bend left=40}](C)(F)
\end{tikzpicture}
}}{\caption{\textbf{End-case \breakdot{2.1.1.2.1.1.1.}} Assume that \textbf{P1} has $BI$. Then \textbf{P2} takes $FI$.\\The potential base is $EF$.}
\label{fig:2.1.1.2.1.1.1}}

\ffigbox{\scalebox{.55}{
\begin{tikzpicture}
  \SetGraphUnit{3}
\GraphInit[vstyle=Normal]
  \SetVertexNormal[Shape      = circle,
                   LineWidth  = 1pt]
  \SetVertexMath
  \Vertex{A}
  \Vertex[x=1,y=1]{B}
  \Vertex[x=0,y=-2]{C}
  \Vertex[x=2,y=-2]{D}
  \Vertex[x=2,y=0]{E}
  \Vertex[x=1,y=2.5]{F}
  \Vertex[x=2,y=1.5]{I}
  
  \SetVertexNormal[Shape = rectangle,
  					LineWidth = 1pt]
  \Vertex[x=2,y=2.5]{+2}
  \Edge(A)(B)
  \Edge(E)(C)
  \Edge(B)(F)
  \Edge(E)(F)
  \Edge(E)(I)
  \Edge(B)(I)
  \SetUpEdge[lw=3pt]
  \Edge(B)(E)
  \SetUpEdge[style={dashed}]
  \Edge(C)(D)
  \Edge(A)(E)
  \Edge(B)(C)
  \Edge(A)(F)
  \Edge[style={dashed, bend left=40}](C)(F)
\end{tikzpicture}
}}{\caption{\textbf{End-case \breakdot{2.1.1.2.1.1.2.}} Assume that \textbf{P1} does not have $BI$. Then \textbf{P2} takes this edge.\\The potential base is $BE$.}
\label{fig:2.1.1.2.1.1.2}}

\ffigbox{
\scalebox{.55}{
\begin{tikzpicture}
  \SetGraphUnit{3}
\GraphInit[vstyle=Normal]
  \SetVertexNormal[Shape      = circle,
                   LineWidth  = 1pt]
  \SetVertexMath
  \Vertex{A}
  \Vertex[x=1,y=1]{B}
  \Vertex[x=0,y=-2]{C}
  \Vertex[x=2,y=-2]{D}
  \Vertex[x=2,y=0]{E}
  \Vertex[x=1,y=2.5]{F}
  
  \SetVertexNormal[Shape = rectangle,
  					LineWidth = 1pt]
  \Vertex[x=2,y=2]{+2}
  \Edge(A)(B)
  \Edge(B)(E)
  \Edge(E)(C)
  \Edge(B)(F)
  \Edge(E)(F)
  \Edge(A)(F)
  \SetUpEdge[style={dashed}]
  \Edge(C)(D)
  \Edge(A)(E)
  \Edge(B)(C)
  \Edge[style={dashed, bend left=40}](C)(F)
\end{tikzpicture}
}}{\caption{\textbf{Split-case \breakdot{2.1.1.2.1.2.}} Assume that \textbf{P1} does not have $AF$. Then \textbf{P2} takes this edge.
}
\label{fig:2.1.1.2.1.2}}
\end{floatrow}
\end{figure}

\begin{figure}[H]
\begin{floatrow}
\ffigbox{\scalebox{.55}{
\begin{tikzpicture}
  \SetGraphUnit{3}
\GraphInit[vstyle=Normal]
  \SetVertexNormal[Shape      = circle,
                   LineWidth  = 1pt]
  \SetVertexMath
  \Vertex{A}
  \Vertex[x=1,y=1]{B}
  \Vertex[x=0,y=-2]{C}
  \Vertex[x=2,y=-2]{D}
  \Vertex[x=2,y=0]{E}
  \Vertex[x=1,y=2.5]{F}
  
  \SetVertexNormal[Shape = rectangle,
  					LineWidth = 1pt]
  \Edge(A)(B)
  \Edge(B)(E)
  \Edge(E)(C)
  \Edge(B)(F)
  \Edge(E)(F)
  \Edge(A)(F)
  \SetUpEdge[style={dashed}]
  \Edge(C)(D)
  \Edge(A)(E)
  \Edge(B)(C)
  \Edge(B)(D)
  \Edge[style={dashed, bend left=40}](C)(F)
  \Edge[style={dashed, bend right=40}](D)(F)
\end{tikzpicture}
}}{\caption{\textbf{Split-case \breakdot{2.1.1.2.1.2.1.}} Assume that \textbf{P1} has both $BD$ and $DF$. Then \textbf{P2} takes $FI$, where $I$ is a free vertex in $K^2$. If \textbf{P1} does not take $BI$, \textbf{P2} takes it and stops. Otherwise, \textbf{P1} takes $BI$, then \textbf{P2} takes $FJ$, where $J$ is a free vertex in $K^2$. If \textbf{P1} does not take $BJ$, \textbf{P2} takes it and stops. Otherwise, \textbf{P1} takes $BJ$, then \textbf{P2} takes $FK$, where $K$ is a free vertex in $K^2$. If \textbf{P1} does not take $BK$, \textbf{P2} takes it and stops.
}
\label{fig:2.1.1.2.1.2.1}}

\ffigbox{
\scalebox{.55}{
\begin{tikzpicture}
  \SetGraphUnit{3}
\GraphInit[vstyle=Normal]
  \SetVertexNormal[Shape      = circle,
                   LineWidth  = 1pt]
  \SetVertexMath
  \Vertex{A}
  \Vertex[x=1,y=1]{B}
  \Vertex[x=0,y=-2]{C}
  \Vertex[x=2,y=-2]{D}
  \Vertex[x=2,y=0]{E}
  \Vertex[x=1,y=2.5]{F}
  \Vertex[x=3.5,y=0.5]{I}
  \Vertex[x=3.5,y=1.5]{J}
  \Vertex[x=3,y=2.5]{K}
  
  \SetVertexNormal[Shape = rectangle,
  					LineWidth = 1pt]
  \Edge(A)(B)
  \Edge(B)(E)
  \Edge(E)(C)
  \Edge(B)(F)
  \Edge(E)(F)
  \Edge(A)(F)
  \Edge(F)(I)
  \Edge(F)(J)
  \Edge(F)(K)
  \SetUpEdge[style={dashed}]
  \Edge(C)(D)
  \Edge(A)(E)
  \Edge(B)(C)
  \Edge(B)(D)
  \Edge(B)(I)
  \Edge(B)(J)
  \Edge(B)(K)
  \Edge[style={dashed, bend left=40}](C)(F)
  \Edge[style={dashed, bend right=40}](D)(F)
\end{tikzpicture}}
\scalebox{.55}{
\begin{tikzpicture}
  \SetGraphUnit{3}
\GraphInit[vstyle=Normal]
  \SetVertexNormal[Shape      = circle,
                   LineWidth  = 1pt]
  \SetVertexMath
  \Vertex{A}
  \Vertex[x=1,y=1]{B}
  \Vertex[x=0,y=-2]{C}
  \Vertex[x=2,y=-2]{D}
  \Vertex[x=2,y=0]{E}
  \Vertex[x=1,y=2.5]{F}
  \Vertex[x=3.5,y=0.5]{I}
  \Vertex[x=3.5,y=1.5]{J}
  \Vertex[x=3,y=2.5]{K}
  
  \SetVertexNormal[Shape = rectangle,
  					LineWidth = 1pt]
  \Vertex[x=0,y=2.5]{+2}
  \Edge(A)(B)
  \Edge(B)(E)
  \Edge(E)(C)
  \Edge(B)(F)
  \Edge(A)(F)
  \Edge(F)(I)
  \Edge(F)(J)
  \Edge(F)(K)
  \Edge(E)(I)
  \Edge(E)(J)
  \SetUpEdge[lw=3pt]
  \Edge(E)(F)
  \SetUpEdge[style={dashed}]
  \Edge(C)(D)
  \Edge(A)(E)
  \Edge(B)(C)
  \Edge(B)(D)
  \Edge(B)(I)
  \Edge(B)(J)
  \Edge(B)(K)
  \Edge[style={dashed, bend left=40}](C)(F)
  \Edge[style={dashed, bend right=40}](D)(F)
\end{tikzpicture}}}{
\caption{\textbf{Special end-case \breakdot{2.1.1.2.1.2.1.1.}} Assume that in the strategy described in the previous case \textbf{P1} did take $BI$, $BJ$ and $BK$ successively. Then \textbf{P2} takes $EI$ and one of $EJ$ and $EK$ which remains free after \textbf{P1}'s turn. Without loss of generality, let this edge be $EJ$.\\This case is treated separately in Section~\ref{s4.4}.}
\label{fig:2.1.1.2.1.2.1.1}}
\end{floatrow}
\end{figure}

\begin{figure}[H]
\begin{floatrow}
\ffigbox{\scalebox{.55}{
\begin{tikzpicture}
  \SetGraphUnit{3}
\GraphInit[vstyle=Normal]
  \SetVertexNormal[Shape      = circle,
                   LineWidth  = 1pt]
  \SetVertexMath
  \Vertex{A}
  \Vertex[x=1,y=1]{B}
  \Vertex[x=0,y=-2]{C}
  \Vertex[x=2,y=-2]{D}
  \Vertex[x=2,y=0]{E}
  \Vertex[x=1,y=2.5]{F}
  \Vertex[x=3.5,y=0.5]{I}
  \Vertex[x=3.5,y=1.5]{J}
  \Vertex[x=3,y=2.5]{K}
  
  \SetVertexNormal[Shape = rectangle,
  					LineWidth = 1pt]
  \Vertex[x=0,y=2.5]{+2}
  \Edge(A)(B)
  \Edge(B)(E)
  \Edge(E)(C)
  \Edge(E)(F)
  \Edge(A)(F)
  \Edge(F)(K)
  \Edge(B)(K)
  \SetUpEdge[lw=3pt]
  \Edge(B)(F)
  \SetUpEdge[style={dashed}]
  \Edge(C)(D)
  \Edge(A)(E)
  \Edge(B)(C)
  \Edge(B)(D)
  \Edge(B)(I)
  \Edge(B)(J)
  \Edge[style={dashed, bend left=40}](C)(F)
  \Edge[style={dashed, bend right=40}](D)(F)
\end{tikzpicture}
}}{\caption{\textbf{Special end-case \breakdot{2.1.1.2.1.2.1.2.}} Assume that in the strategy described in split-case \textbf{\breakdot{2.1.1.2.1.2.1.}} \textbf{P1} did not take $BI$, then $BJ$ and then $BK$. Without loss of generality we can assume that $BK$ and $FK$ were taken by \textbf{P2}, provided that \textbf{P2} concede both $FI$ and $FJ$ if they were taken by \textbf{P2} and that we give \textbf{P1} the edges $BI$ and $BJ$ regardless if \textbf{P1} took them (this corresponds to the case where \textbf{P1} responded the first two times, but not the third, and \textbf{P2} promises not to use the first two edges he acquired).\\This case is treated separately in Section~\ref{s4.4}.}
\label{fig:2.1.1.2.1.2.1.2}}

\ffigbox{
\scalebox{.55}{
\begin{tikzpicture}
  \SetGraphUnit{3}
\GraphInit[vstyle=Normal]
  \SetVertexNormal[Shape      = circle,
                   LineWidth  = 1pt]
  \SetVertexMath
  \Vertex{A}
  \Vertex[x=1,y=1]{B}
  \Vertex[x=0,y=-2]{C}
  \Vertex[x=2,y=-2]{D}
  \Vertex[x=2,y=0]{E}
  \Vertex[x=1,y=2.5]{F}
  
  \SetVertexNormal[Shape = rectangle,
  					LineWidth = 1pt]
  \Vertex[x=2,y=2]{+2}
  \Edge(A)(B)
  \Edge(B)(E)
  \Edge(E)(C)
  \Edge(E)(F)
  \Edge(A)(F)
  \SetUpEdge[lw=3pt]
  \Edge(B)(F)
  \SetUpEdge[style={dashed}]
  \Edge(C)(D)
  \Edge(A)(E)
  \Edge(B)(C)
  \Edge[style={dashed, bend left=40}](C)(F)
\end{tikzpicture}
}}{\caption{\textbf{End-case \breakdot{2.1.1.2.1.2.2.}} Assume that \textbf{P1} does not have both $BD$ and $DF$, so that \textbf{P1} does not have a $2\Delta$-configuration for $BF$.\\
The potential base is $BF$.
}
\label{fig:2.1.1.2.1.2.2}}
\end{floatrow}
\end{figure}

\begin{figure}[H]
\begin{floatrow}[3]
\ffigbox{
\scalebox{.55}{
\begin{tikzpicture}
  \SetGraphUnit{3}
\GraphInit[vstyle=Normal]
  \SetVertexNormal[Shape      = circle,
                   LineWidth  = 1pt]
  \SetVertexMath
  \Vertex{A}
  \Vertex[x=1,y=1]{B}
  \Vertex[x=0,y=-2]{C}
  \Vertex[x=2,y=-2]{D}
  \Vertex[x=2,y=0]{E}
  \Vertex[x=1,y=2.5]{F}
  
  \SetVertexNormal[Shape = rectangle,
  					LineWidth = 1pt]
  \Vertex[x=2,y=2]{+3}
  \Edge(A)(B)
  \Edge(B)(E)
  \Edge(E)(C)
  \Edge(B)(F)
  \Edge[style={bend left=40}](C)(F)
  \SetUpEdge[lw=3pt]
  \Edge(E)(F)
  \SetUpEdge[style={dashed}]
  \Edge(C)(D)
  \Edge(A)(E)
  \Edge(B)(C)
\end{tikzpicture}
}}{\caption{\textbf{End-case \breakdot{2.1.1.2.2.}} Assume that \textbf{P1} does not have $CF$. Then \textbf{P2} takes this edge.\\The potential base is $EF$.}
\label{fig:2.1.1.2.2}}

\ffigbox{\scalebox{.55}{
\begin{tikzpicture}
  \SetGraphUnit{3}
\GraphInit[vstyle=Normal]
  \SetVertexNormal[Shape      = circle,
                   LineWidth  = 1pt]
  \SetVertexMath
  \Vertex{A}
  \Vertex[x=1,y=1]{B}
  \Vertex[x=0,y=-2]{C}
  \Vertex[x=2,y=-2]{D}
  \Vertex[x=2,y=0]{E}
  
  \SetVertexNormal[Shape = rectangle,
  					LineWidth = 1pt]
  \Vertex[x=0,y=1.5]{+2}
  \Edge(A)(B)
  \Edge(B)(E)
  \Edge(E)(C)
  \Edge(B)(C)
  \SetUpEdge[style={dashed}]
  \Edge(C)(D)
  \Edge(A)(E)
\end{tikzpicture}
}}{\caption{\textbf{Split-case \breakdot{2.1.2.}} Assume that \textbf{P1} does not have $BC$. Then \textbf{P2} takes this edge.}
\label{fig:2.1.2}}

\ffigbox{\scalebox{.55}{
\begin{tikzpicture}
  \SetGraphUnit{3}
\GraphInit[vstyle=Normal]
  \SetVertexNormal[Shape      = circle,
                   LineWidth  = 1pt]
  \SetVertexMath
  \Vertex{A}
  \Vertex[x=1,y=1]{B}
  \Vertex[x=0,y=-2]{C}
  \Vertex[x=2,y=-2]{D}
  \Vertex[x=2,y=0]{E}
  
  \SetVertexNormal[Shape = rectangle,
  					LineWidth = 1pt]
  \Vertex[x=0,y=1.5]{+1}
  \Edge(A)(B)
  \Edge(B)(E)
  \Edge(E)(C)
  \Edge(B)(C)
  \SetUpEdge[style={dashed}]
  \Edge(C)(D)
  \Edge(A)(E)
  \Edge(A)(C)
\end{tikzpicture}
}}{\caption{\textbf{Split-case \breakdot{2.1.2.1.}} Assume that \textbf{P1} has $AC$.}
\label{fig:2.1.2.1}}
\end{floatrow}
\end{figure}

\begin{figure}[H]
\begin{floatrow}[3]
\ffigbox{
\scalebox{.55}{
\begin{tikzpicture}
  \SetGraphUnit{3}
\GraphInit[vstyle=Normal]
  \SetVertexNormal[Shape      = circle,
                   LineWidth  = 1pt]
  \SetVertexMath
  \Vertex{A}
  \Vertex[x=1,y=1]{B}
  \Vertex[x=0,y=-2]{C}
  \Vertex[x=2,y=-2]{D}
  \Vertex[x=2,y=0]{E}
  
  \SetVertexNormal[Shape = rectangle,
  					LineWidth = 1pt]
  \Edge(A)(B)
  \Edge(B)(E)
  \Edge(E)(C)
  \Edge(B)(C)
  \SetUpEdge[style={dashed}]
  \Edge(C)(D)
  \Edge(A)(E)
  \Edge(A)(C)
  \Edge(D)(E)
\end{tikzpicture}
}
\scalebox{.55}{
\begin{tikzpicture}
  \SetGraphUnit{3}
\GraphInit[vstyle=Normal]
  \SetVertexNormal[Shape      = circle,
                   LineWidth  = 1pt]
  \SetVertexMath
  \Vertex{A}
  \Vertex[x=1,y=1]{B}
  \Vertex[x=0,y=-2]{C}
  \Vertex[x=2,y=-2]{D}
  \Vertex[x=2,y=0]{E}
  \Vertex[x=2,y=1.5]{F}
  
  \SetVertexNormal[Shape = rectangle,
  					LineWidth = 1pt]
  \Vertex[x=0,y=1.5]{+2}
  \Edge(A)(B)
  \Edge(B)(E)
  \Edge(E)(C)
  \Edge(B)(C)
  \Edge(A)(D)
  \Edge(B)(F)
  \SetUpEdge[style={dashed}]
  \Edge(C)(D)
  \Edge(A)(E)
  \Edge(A)(C)
  \Edge(D)(E)
\end{tikzpicture}
}}{\caption{\textbf{Marked split-case \breakdot{2.1.2.1.1.}} Assume that \textbf{P1} has $DE$. Then \textbf{P2} takes $AD$ and then $BF$, where $F$ is a free vertex in $K^2$.\\\textbf{P1} lost $2$ edges.}
\label{fig:2.1.2.1.1}}

\ffigbox{\scalebox{.55}{
\begin{tikzpicture}
  \SetGraphUnit{3}
\GraphInit[vstyle=Normal]
  \SetVertexNormal[Shape      = circle,
                   LineWidth  = 1pt]
  \SetVertexMath
  \Vertex{A}
  \Vertex[x=1,y=1]{B}
  \Vertex[x=0,y=-2]{C}
  \Vertex[x=2,y=-2]{D}
  \Vertex[x=2,y=0]{E}
  \Vertex[x=2,y=1.5]{F}
  
  \SetVertexNormal[Shape = rectangle,
  					LineWidth = 1pt]
  \Vertex[x=0,y=1.5]{+1}
  \Edge(A)(B)
  \Edge(B)(E)
  \Edge(E)(C)
  \Edge(A)(D)
  \Edge(B)(F)
  \Edge(B)(C)
  \SetUpEdge[style={dashed}]
  \Edge(C)(D)
  \Edge(A)(E)
  \Edge(A)(C)
  \Edge(D)(E)
  \Edge(E)(F)
\end{tikzpicture}
}
\scalebox{.55}{
\begin{tikzpicture}
  \SetGraphUnit{3}
\GraphInit[vstyle=Normal]
  \SetVertexNormal[Shape      = circle,
                   LineWidth  = 1pt]
  \SetVertexMath
  \Vertex{A}
  \Vertex[x=1,y=1]{B}
  \Vertex[x=0,y=-2]{C}
  \Vertex[x=2,y=-2]{D}
  \Vertex[x=2,y=0]{E}
  \Vertex[x=2,y=1.5]{F}
  
  \SetVertexNormal[Shape = rectangle,
  					LineWidth = 1pt]
  \Vertex[x=0,y=1.5]{+2}
  \Edge(A)(B)
  \Edge(B)(E)
  \Edge(E)(C)
  \Edge(A)(D)
  \Edge(B)(F)
  \Edge(C)(F)
  \SetUpEdge[lw=3pt]
  \Edge(B)(C)
  \SetUpEdge[style={dashed}]
  \Edge(C)(D)
  \Edge(A)(E)
  \Edge(A)(C)
  \Edge(D)(E)
  \Edge(E)(F)
\end{tikzpicture}
}}{\caption{\textbf{End-case \breakdot{2.1.2.1.1.1.}} Assume that \textbf{P1} has $EF$. Then he does not have $CF$, as $F$ was a \textbf{P1}-free vertex before his move. Then \textbf{P2} takes $CF$.\\The potential base is $BC$.}
\label{fig:2.1.2.1.1.1}}

\ffigbox{\scalebox{.55}{
\begin{tikzpicture}
  \SetGraphUnit{3}
\GraphInit[vstyle=Normal]
  \SetVertexNormal[Shape      = circle,
                   LineWidth  = 1pt]
  \SetVertexMath
  \Vertex{A}
  \Vertex[x=1,y=1]{B}
  \Vertex[x=0,y=-2]{C}
  \Vertex[x=2,y=-2]{D}
  \Vertex[x=2,y=0]{E}
  \Vertex[x=2,y=1.5]{F}
  
  \SetVertexNormal[Shape = rectangle,
  					LineWidth = 1pt]
  \Vertex[x=0,y=1.5]{+3}
  \Edge(A)(B)
  \Edge(E)(C)
  \Edge(B)(C)
  \Edge(A)(D)
  \Edge(B)(F)
  \Edge(E)(F)
  \SetUpEdge[lw=3pt]
  \Edge(B)(E)
  \SetUpEdge[style={dashed}]
  \Edge(C)(D)
  \Edge(A)(E)
  \Edge(A)(C)
  \Edge(D)(E)
\end{tikzpicture}
}}{\caption{\textbf{End-case \breakdot{2.1.2.1.1.2.}} Assume that \textbf{P1} does not have $EF$. Then \textbf{P2} takes this edge.\\The potential base is $BE$.}
\label{fig:2.1.2.1.1.2}}
\end{floatrow}
\end{figure}

\begin{figure}[H]
\begin{floatrow}[3]
\ffigbox{
\scalebox{.55}{
\begin{tikzpicture}
  \SetGraphUnit{3}
\GraphInit[vstyle=Normal]
  \SetVertexNormal[Shape      = circle,
                   LineWidth  = 1pt]
  \SetVertexMath
  \Vertex{A}
  \Vertex[x=1,y=1]{B}
  \Vertex[x=0,y=-2]{C}
  \Vertex[x=2,y=-2]{D}
  \Vertex[x=2,y=0]{E}
  
  \SetVertexNormal[Shape = rectangle,
  					LineWidth = 1pt]
  \Vertex[x=0,y=1.5]{+2}
  \Edge(A)(B)
  \Edge(B)(E)
  \Edge(E)(C)
  \Edge(B)(C)
  \Edge(D)(E)
  \SetUpEdge[style={dashed}]
  \Edge(C)(D)
  \Edge(A)(E)
  \Edge(A)(C)
\end{tikzpicture}
}}{\caption{\textbf{Marked split-case \breakdot{2.1.2.1.2.}} Assume that \textbf{P1} does not have $DE$. Then \textbf{P2} takes this edge.\\\textbf{P1} lost an edge.}
\label{fig:2.1.2.1.2}}

\ffigbox{\scalebox{.55}{
\begin{tikzpicture}
  \SetGraphUnit{3}
\GraphInit[vstyle=Normal]
  \SetVertexNormal[Shape      = circle,
                   LineWidth  = 1pt]
  \SetVertexMath
  \Vertex{A}
  \Vertex[x=1,y=1]{B}
  \Vertex[x=0,y=-2]{C}
  \Vertex[x=2,y=-2]{D}
  \Vertex[x=2,y=0]{E}
  
  \SetVertexNormal[Shape = rectangle,
  					LineWidth = 1pt]
  \Vertex[x=0,y=1.5]{+1}
  \Edge(A)(B)
  \Edge(B)(E)
  \Edge(E)(C)
  \Edge(B)(C)
  \Edge(D)(E)
  \SetUpEdge[style={dashed}]
  \Edge(C)(D)
  \Edge(A)(E)
  \Edge(A)(C)
  \Edge(B)(D)
\end{tikzpicture}
}
\scalebox{.55}{
\begin{tikzpicture}
  \SetGraphUnit{3}
\GraphInit[vstyle=Normal]
  \SetVertexNormal[Shape      = circle,
                   LineWidth  = 1pt]
  \SetVertexMath
  \Vertex{A}
  \Vertex[x=1,y=1]{B}
  \Vertex[x=0,y=-2]{C}
  \Vertex[x=2,y=-2]{D}
  \Vertex[x=2,y=0]{E}
  \Vertex[x=2,y=1.5]{F}
  
  \SetVertexNormal[Shape = rectangle,
  					LineWidth = 1pt]
  \Vertex[x=0,y=1.5]{+2}
  \Edge(A)(B)
  \Edge(B)(E)
  \Edge(E)(C)
  \Edge(D)(E)
  \Edge(B)(F)
  \Edge(B)(C)
  \SetUpEdge[style={dashed}]
  \Edge(C)(D)
  \Edge(A)(E)
  \Edge(A)(C)
  \Edge(B)(D)
\end{tikzpicture}
}}{\caption{\textbf{Marked split-case \breakdot{2.1.2.1.2.1.}} Assume that \textbf{P1} has $BD$. Then \textbf{P2} takes $BF$, where $F$ is a free vertex in $K^2$.\\\textbf{P1} lost $2$ edges.}
\label{fig:2.1.2.1.2.1}}

\ffigbox{\scalebox{.55}{\begin{tikzpicture}
  \SetGraphUnit{3}
\GraphInit[vstyle=Normal]
  \SetVertexNormal[Shape      = circle,
                   LineWidth  = 1pt]
  \SetVertexMath
  \Vertex{A}
  \Vertex[x=1,y=1]{B}
  \Vertex[x=0,y=-2]{C}
  \Vertex[x=2,y=-2]{D}
  \Vertex[x=2,y=0]{E}
  \Vertex[x=2,y=1.5]{F}
  
  \SetVertexNormal[Shape = rectangle,
  					LineWidth = 1pt]
  \Vertex[x=0,y=1.5]{+1}
  \Edge(A)(B)
  \Edge(B)(E)
  \Edge(E)(C)
  \Edge(D)(E)
  \Edge(B)(F)
  \Edge(B)(C)
  \SetUpEdge[style={dashed}]
  \Edge(C)(D)
  \Edge(A)(E)
  \Edge(A)(C)
  \Edge(B)(D)
  \Edge(E)(F)
\end{tikzpicture}
}
\scalebox{.55}{
\begin{tikzpicture}
  \SetGraphUnit{3}
\GraphInit[vstyle=Normal]
  \SetVertexNormal[Shape      = circle,
                   LineWidth  = 1pt]
  \SetVertexMath
  \Vertex{A}
  \Vertex[x=1,y=1]{B}
  \Vertex[x=0,y=-2]{C}
  \Vertex[x=2,y=-2]{D}
  \Vertex[x=2,y=0]{E}
  \Vertex[x=2,y=1.5]{F}
  
  \SetVertexNormal[Shape = rectangle,
  					LineWidth = 1pt]
  \Vertex[x=0,y=1.5]{+2}
  \Edge(A)(B)
  \Edge(B)(E)
  \Edge(E)(C)
  \Edge(D)(E)
  \Edge(B)(F)
  \Edge(C)(F)
  \SetUpEdge[lw=3pt]
  \Edge(B)(C)
  \SetUpEdge[style={dashed}]
  \Edge(C)(D)
  \Edge(A)(E)
  \Edge(A)(C)
  \Edge(B)(D)
  \Edge(E)(F)
\end{tikzpicture}
}}{\caption{\textbf{End-case \breakdot{2.1.2.1.2.1.1.}} Assume that \textbf{P1} has $EF$. Then he does not have $CF$, as $F$ was a \textbf{P1}-free vertex before his move. Then \textbf{P2} takes $CF$.\\The potential base is $BC$.}
\label{fig:2.1.2.1.2.1.1}}
\end{floatrow}
\end{figure}

\begin{figure}[H]
\begin{floatrow}[3]
\ffigbox{
\scalebox{.55}{
\begin{tikzpicture}
  \SetGraphUnit{3}
\GraphInit[vstyle=Normal]
  \SetVertexNormal[Shape      = circle,
                   LineWidth  = 1pt]
  \SetVertexMath
  \Vertex{A}
  \Vertex[x=1,y=1]{B}
  \Vertex[x=0,y=-2]{C}
  \Vertex[x=2,y=-2]{D}
  \Vertex[x=2,y=0]{E}
  \Vertex[x=2,y=1.5]{F}
  
  \SetVertexNormal[Shape = rectangle,
  					LineWidth = 1pt]
  \Vertex[x=0,y=1.5]{+3}
  \Edge(A)(B)
  \Edge(E)(C)
  \Edge(B)(C)
  \Edge(D)(E)
  \Edge(B)(F)
  \Edge(E)(F)
  \SetUpEdge[lw=3pt]
  \Edge(B)(E)
  \SetUpEdge[style={dashed}]
  \Edge(C)(D)
  \Edge(A)(E)
  \Edge(A)(C)
  \Edge(B)(D)
\end{tikzpicture}
}}{\caption{\textbf{End-case \breakdot{2.1.2.1.2.1.2.}} Assume that \textbf{P1} does not have $EF$. Then \textbf{P2} takes this edge.\\The potential base is $BE$.}
\label{fig:2.1.2.1.2.1.2}}

\ffigbox{\scalebox{.55}{
\begin{tikzpicture}
  \SetGraphUnit{3}
\GraphInit[vstyle=Normal]
  \SetVertexNormal[Shape      = circle,
                   LineWidth  = 1pt]
  \SetVertexMath
  \Vertex{A}
  \Vertex[x=1,y=1]{B}
  \Vertex[x=0,y=-2]{C}
  \Vertex[x=2,y=-2]{D}
  \Vertex[x=2,y=0]{E}
  
  \SetVertexNormal[Shape = rectangle,
  					LineWidth = 1pt]
  \Vertex[x=0,y=1.5]{+3}
  \Edge(A)(B)
  \Edge(E)(C)
  \Edge(B)(C)
  \Edge(D)(E)
  \Edge(B)(D)
  \SetUpEdge[lw=3pt]
  \Edge(B)(E)
  \SetUpEdge[style={dashed}]
  \Edge(C)(D)
  \Edge(A)(E)
  \Edge(A)(C)
\end{tikzpicture}
}}{\caption{\textbf{End-case \breakdot{2.1.2.1.2.2.}} Assume that \textbf{P1} does not have $BD$. Then \textbf{P2} takes this edge.\\The potential base is $BE$.}
\label{fig:2.1.2.1.2.2}}

\ffigbox{\scalebox{.55}{
\begin{tikzpicture}
  \SetGraphUnit{3}
\GraphInit[vstyle=Normal]
  \SetVertexNormal[Shape      = circle,
                   LineWidth  = 1pt]
  \SetVertexMath
  \Vertex{A}
  \Vertex[x=1,y=1]{B}
  \Vertex[x=0,y=-2]{C}
  \Vertex[x=2,y=-2]{D}
  \Vertex[x=2,y=0]{E}
  
  \SetVertexNormal[Shape = rectangle,
  					LineWidth = 1pt]
  \Vertex[x=0,y=1.5]{+3}
  \Edge(A)(B)
  \Edge(B)(E)
  \Edge(E)(C)
  \Edge(A)(C)
  \SetUpEdge[lw=3pt]
  \Edge(B)(C)
  \SetUpEdge[style={dashed}]
  \Edge(C)(D)
  \Edge(A)(E)
\end{tikzpicture}
}}{\caption{\textbf{End-case \breakdot{2.1.2.2.}} Assume \textbf{P1} does not have $AC$. Then \textbf{P2} takes this edge.\\The potential base is $BC$.}
\label{fig:2.1.2.2}}
\end{floatrow}
\end{figure}

\begin{figure}[H]
\begin{floatrow}[3]
\ffigbox{
\scalebox{.55}{
\begin{tikzpicture}
  \SetGraphUnit{3}
\GraphInit[vstyle=Normal]
  \SetVertexNormal[Shape      = circle,
                   LineWidth  = 1pt]
  \SetVertexMath
  \Vertex{A}
  \Vertex[x=1,y=1]{B}
  \Vertex[x=0,y=-2]{C}
  \Vertex[x=2,y=-2]{D}
  \Vertex[x=2,y=0]{E}
  
  \SetVertexNormal[Shape = rectangle,
  					LineWidth = 1pt]
  \Vertex[x=0,y=1.5]{+2}
  \Edge(A)(B)
  \Edge(B)(E)fig
  \Edge(A)(E)
  \SetUpEdge[style={dashed}]
  \Edge(C)(D)
\end{tikzpicture}
}
\scalebox{.55}{
\begin{tikzpicture}
  \SetGraphUnit{3}
\GraphInit[vstyle=Normal]
  \SetVertexNormal[Shape      = circle,
                   LineWidth  = 1pt]
  \SetVertexMath
  \Vertex{A}
  \Vertex[x=1,y=1]{B}
  \Vertex[x=0,y=-2]{C}
  \Vertex[x=2,y=-2]{D}
  \Vertex[x=2,y=0]{E}
  \Vertex[x=1,y=2.5]{F}
  
  \SetVertexNormal[Shape = rectangle,
  					LineWidth = 1pt]
  \Vertex[x=0,y=2]{+3}
  \Edge(B)(E)
  \Edge(A)(E)
  \Edge(B)(F)
  \Edge(A)(B)
  \SetUpEdge[style={dashed}]
  \Edge(C)(D)
\end{tikzpicture}
}}{\caption{\textbf{Split-case 2.2.} Assume that \textbf{P1} does not have $AE$. Then \textbf{P2} takes this edge and plays from a \textbf{P1}-free vertex among $A$, $B$ and $E$ to a free vertex in $K^2$, say $F$. Without loss of generality assume \textbf{P2} played from $B$. Hence, after \textbf{P1}'s turn there is a vertex among $A$ and $E$ of \textbf{P1}-degree at most $1$.}
\label{fig:2.2}}

\ffigbox{\scalebox{.55}{
\begin{tikzpicture}
  \SetGraphUnit{3}
\GraphInit[vstyle=Normal]
  \SetVertexNormal[Shape      = circle,
                   LineWidth  = 1pt]
  \SetVertexMath
  \Vertex{A}
  \Vertex[x=1,y=1]{B}
  \Vertex[x=0,y=-2]{C}
  \Vertex[x=2,y=-2]{D}
  \Vertex[x=2,y=0]{E}
  \Vertex[x=1,y=2.5]{F}
  
  \SetVertexNormal[Shape = rectangle,
  					LineWidth = 1pt]
  \Vertex[x=0,y=1.5]{+2}
  \Edge(A)(B)
  \Edge(B)(E)
  \Edge(A)(E)
  \Edge(B)(F)
  \SetUpEdge[style={dashed}]
  \Edge(C)(D)
  \Edge(E)(F)
\end{tikzpicture}
}\quad
\scalebox{.55}{
\begin{tikzpicture}
  \SetGraphUnit{3}
\GraphInit[vstyle=Normal]
  \SetVertexNormal[Shape      = circle,
                   LineWidth  = 1pt]
  \SetVertexMath
  \Vertex{A}
  \Vertex[x=1,y=1]{B}
  \Vertex[x=0,y=-2]{C}
  \Vertex[x=2,y=-2]{D}
  \Vertex[x=2,y=0]{E}
  \Vertex[x=1,y=2.5]{F}
  
  \SetVertexNormal[Shape = rectangle,
  					LineWidth = 1pt]
  \Vertex[x=0,y=2]{+3}
  \Edge(B)(E)
  \Edge(A)(E)
  \Edge(A)(F)
  \Edge(B)(F)
  \SetUpEdge[lw=3pt]
  \Edge(A)(B)
  \SetUpEdge[style={dashed}]
  \Edge(C)(D)
  \Edge(E)(F)
\end{tikzpicture}
}}{\caption{\textbf{End-case \breakdot{2.2.1.}} Assume \textbf{P1} has $AF$ or $EF$. Without loss of generality assume \textbf{P1} has $EF$. Then he does not have $AF$, as $F$ was a \textbf{P1}-free vertex before his last turn. Then \textbf{P2} takes $AF$.\\The potential base is $AB$.}
\label{fig:2.2.1}}

\ffigbox{\scalebox{.55}{
\begin{tikzpicture}
  \SetGraphUnit{3}
\GraphInit[vstyle=Normal]
  \SetVertexNormal[Shape      = circle,
                   LineWidth  = 1pt]
  \SetVertexMath
  \Vertex{A}
  \Vertex[x=1,y=1]{B}
  \Vertex[x=0,y=-2]{C}
  \Vertex[x=2,y=-2]{D}
  \Vertex[x=2,y=0]{E}
  \Vertex[x=1,y=2.5]{F}
  
  \SetVertexNormal[Shape = rectangle,
  					LineWidth = 1pt]
  \Vertex[x=0,y=1.5]{+4}
  \Edge(B)(E)
  \Edge(A)(E)
  \Edge(A)(F)
  \Edge(B)(F)
  \SetUpEdge[lw=3pt]
  \Edge(A)(B)
  \SetUpEdge[style={dashed}]
  \Edge(C)(D)
\end{tikzpicture}
}}{\caption{\textbf{End-case \breakdot{2.2.2.}} Assume \textbf{P1} does not have any of $AF$ and $EF$. Without loss of generality assume $\deg_{\textbf{P1}}(A)\le 1$. Then \textbf{P2} takes $AF$.\\The potential base is $AB$, as $\deg_\textbf{P1}(A)+\deg_{\textbf{P1}}(B) \leq 1+2$.}
\label{fig:2.2.2}}
\end{floatrow}
\end{figure}

\subsection{Special end-cases}
\label{s4.4}
In this section we present a drawing strategy for \textbf{P2} in the two special end-cases, which cannot be dealt with by directly applying Lemma~\ref{lem:2}.

\paragraph{End-case \breakdot{2.1.1.2.1.2.1.1.} Figure~\ref{fig:2.1.1.2.1.2.1.1}-right} Note that \textbf{P1} lost $3$ edges (i.e. he has not won yet). If \textbf{P1} does not have $EK$, \textbf{P2} takes it and creates a $G$ before him. Otherwise, one of \textbf{P1}'s additional edges is $EK$ and hence \textbf{P1} lost $4$ edges. Moreover, by inspection we observe that \textbf{P1} has no 4-cycles $E, C_1, C_2, C_3$ with the edge $EC_2$ not taken by \textbf{P2} and at most one triangle $ET_1T_2$. Let $\mathcal{E}_0$ be the initial set of edges taken by \textbf{P1} in $K^2$, so that $|\mathcal{E}_0|\le11$. Then \textbf{P2} takes edges from $F$ to free vertices $L_i$ for $i=1,\dots,k$ until at some point \textbf{P1} does not take the edge $EL_k$. Then \textbf{P2} takes $EL_k$ and completes a $G$. The game stops after \textbf{P2}'s move. Let $\mathcal{E}_1=\{EL_i,1\le i<k\}$ be the star taken by \textbf{P1} and $\mathcal{E}_2$ be his last edge. We claim that \textbf{P1} does not have a $G$. Indeed, \textbf{P1} cannot have a $G$ that does not intersect $\mathcal{E}_1$, as $e_\textbf{P1}(G)\le12-4$ since \textbf{P1} lost $4$ edges and $|\mathcal{E}_0|+|\mathcal{E}_2|\le 12$. On the other hand, \textbf{P1} cannot have a $G$ that contains some $EL_i$, as neither of the vertices $E$ and $L_i$ can be the base. This is because \textbf{P1} has at most $2$ triangles containing $E$ by the above observation and the fact that $\deg_\textbf{P1}(L_i)\le 2$.

\paragraph{End-case \breakdot{2.1.1.2.1.2.1.2.} Figure~\ref{fig:2.1.1.2.1.2.1.2}} Note that \textbf{P1} lost $3$ edges.  (i.e. he has not won yet). Moreover, by inspection we observe that \textbf{P1} has either at most one 4-cycle $F, C_1, C_2, C_3$ with the edge $FC_2$ not taken by anyone and at most one triangle $FT_1T_2$, or no such 4-cycle and at most $2$ such triangles. Let $\mathcal{E}_0$ be the initial set of edges taken by \textbf{P1} in $K^2$, so that $|\mathcal{E}_0|=10$.  Then \textbf{P2} takes edges from $B$ to free vertices $L_i$ for $i=1,\dots,k$ until at some point \textbf{P1} does not take the edge $FL_k$. Then \textbf{P2} takes $FL_k$ and completes a $G$. The game stops after \textbf{P2}'s move. Let $\mathcal{E}_1=\{FL_i,1\le i<k\}$ be the star taken by \textbf{P1} and $\mathcal{E}_2$ be his last edge. We claim that \textbf{P1} does not have a $G$. Indeed, \textbf{P1} cannot have a $G$ that does not intersect $\mathcal{E}_1$, as $e_\textbf{P1}(G)\le 11-3$ since \textbf{P1} lost $3$ edges and $|\mathcal{E}_0|+|\mathcal{E}_2|\le 11$. On the other hand, \textbf{P1} cannot have a $G$ that contains some $FL_i$, as neither of the vertices $F$ and $L_i$ can be the base. This is because \textbf{P1} has at most $3$ triangles containing $F$ by the above observation and the fact that $\deg_\textbf{P1}(L_i)\le 2$.

This analysis concludes the proof of Theorem~\ref{th:main}.\qed

\section{From graphs on two copies to hypergaphs}
\label{sechyp}
In this section we show that $\mathcal{R}(K_\omega^{(4)},G')$ is a draw. Recall from Definition~\ref{def:G'} that $G'$ is the $4$-uniform hypergraph obtained from the graph $G$ by adding $2$ new vertices $X, Y$ and including them in all the edges. We call $X,Y$ the two \emph{centres}. Note that every isomorphism of $G'$ fixes $\{X, Y\}$ and by abuse of notation refer to any copy of $G'$ in $K_\omega^{(4)}$ by $G'$. Given the two points $\{X,Y\}$ the set of all hyperedges containing these points naturally identifies with the set of all edges of $K_{\omega}^{(2)}$. We call this set of hyperedges the \emph{$XY$ board}, which corresponds to the $K_\omega^{X,Y}$ notation from Section~\ref{sec:overview}. We use the vocabulary introduced for $G$ on the $XY$ board and we denote hyperedges on the $XY$ board by their other two vertices when $X,Y$ are clear from the context. Note that the two centres have individual and joint degrees $9=|E(G)|=|E(G')|$. The two points in the base have (individual) degree $5$. The other vertices have degree $2$. For the sake of conciseness, we use the immediate extensions of definitions for graphs to hypergraphs.

\begin{proof}[Proof of Theorem \ref{th:hyper}]
We give the following explicit drawing strategy for \textbf{P2} divided in four stages.

In the first stage \textbf{P2} plays as follows. Without loss of generality \textbf{P1} takes the hyperedge $TUVW$. \textbf{P2} takes the hyperedge $XYAB$, where all vertices are free. \textbf{P1} takes a hyperedge which without loss of generality does not contain $X$.
\textbf{P2} takes the hyperedge $XYBC$ where $C$ is a free vertex and then \textbf{P1} takes a hyperedge. Then \textbf{P1} does not have both $XAYC$ and $XABC$, as $C$ was a \textbf{P1}-free vertex before his move, say he does not have $XAYC$. Then \textbf{P2} takes $XAYC$ and \textbf{P1} takes another hyperedge. At this point on the $XY$ board \textbf{P2} has the hyperedges $AB$, $BC$, $AC$ and \textbf{P1} has at most $2$ hyperedges. \textbf{P2} takes the hyperedge $CD$, where $D$ is a free vertex and \textbf{P1} takes another hyperedge. Note that \textbf{P1} cannot have both $DA$ and $DB$, as $D$ was a \textbf{P1}-free vertex before his move, say he does not have $DA$. Then \textbf{P2} takes $DA$ and \textbf{P1} takes another hyperedge. At the end of the first stage \textbf{P1} has $6$ hyperedges $\mathcal{E}_0$ in total including $TUVW$ and at most $4$ of them are on the $XY$ board.\\

\noindent\textbf{Claim.} At the end of the first stage $AC$ is a potential base $AC=A_0A_1$ (see Figure~\ref{fig:1.2.2}) with special vertex $A_0$ such that on the entire $4$-uniform board \textbf{P1} has at most $2$ hyperedges that contain $A_0$ but do not contain $A_1$.

\begin{proof}[Proof of the claim]
If both vertices $A$ and $C$ can be the special vertex of $AC$, it suffices to choose $A_0$ to be one of them with at most $2$ hyperedges taken by \textbf{P1} containing it but not the other by the pigeon-hole principle (as $TUVW$ is disjoint from $AC$). Otherwise, if, say, $C$ is not special, then none of the hyperedges of the triangle or $4$-cycle with a free hyperedge from $C$ can contain $A$, so \textbf{P1} has at most $2$ hyperedges containing $A$ but not $C$. Therefore, $A$ is the special vertex of the potential base $AC$.\end{proof}

In the second stage we consider the $XY$ board and proceed as in the proof of Lemma~\ref{lem:2}, though the proof requires more attention. \textbf{P2} takes edges from $A_1$ to free vertices $F_i$ for $i=1,\dots,k$ until at some point \textbf{P1} does not take the edge $A_0F_k$. Then \textbf{P2} takes $A_0F_k$ and \textbf{P1} takes two extra hyperedges. Let $\mathcal{E}_{1}=\{A_0F_i,1\le i<k\}$ be the star taken by \textbf{P1} and $\mathcal{E}_2$ be the set consisting of the last two hyperedges taken by \textbf{P1}.

At the end of the second stage \textbf{P1} has $TUVW$, $7$ other hyperedges and the star $\mathcal{E}_1$ (which does not have vertices in common with $TUVW$). For the third stage, consider three cases.\\

\noindent\textbf{Case I.} \textbf{P1} has at least $8$ hyperedges in a copy $G^1$ of $G'$ containing $TUVW$ with its base present.

Assume without loss of generality that $G^1$ has centres $TU$ and base $C_0C_1$. Note that $G^1\cap\mathcal{E}_1=\varnothing$ and so \textbf{P1} has exactly $8$ hyperedges in $G^1$ and those are precisely $\mathcal{E}_0\cup\mathcal{E}_2$. Let $C_{\epsilon_1}D_1$ be the edge absent in $G^1$. Then \textbf{P2} takes the edge $C_{1-\epsilon_1}D_1$ and, while \textbf{P1} keeps taking edges $C_{\epsilon_i}D_i$ to some vertices $D_i$, \textbf{P2} keeps taking edges $C_{1-\epsilon_i}D_i$ for $\epsilon_i \in \{0,1\}$ and $2\le i\le m$ until at some point \textbf{P1} either stops or takes a different kind of edge. Let $\mathcal{E}_3=\{C_{\epsilon_i}D_i,1\le i\le m\}$ and let $\mathcal{E}_4$ be the set consisting of the last edge.

In the fourth stage \textbf{P2} moves back to the $XY$ board. He takes edges from $A_1$ to free vertices $F_{i}$ for $k<i\le l$ until at some point \textbf{P1} does not take the edge $A_0F_{l}$. Then \textbf{P2} takes $A_0F_l$ and constructs a $G$. The game stops after \textbf{P2}'s move. Let $\mathcal{E}_5=\{A_0F_i,k<i<l\}$ and $\mathcal{E}_6$ be the set consisting of the last edge taken by \textbf{P1}. We claim that \textbf{P1} does not have a $G'$.

Assume for a contradiction that \textbf{P1} has a $G'$ with centres $HI$. Note that the $XY$ and $TU$ boards intersect only in the hyperedge $XYTU$, which is not of the form $XYA_0F_j$. Then on the $HI$ board, \textbf{P1} has either at most $1$ hyperedge of the $TU$ board or no hyperedges of the form $XYA_0F_j$. Indeed if, on the contrary, the $HI$ board contains $TUJ_1K_1$, $TUJ_2K_2$ and $XYA_0F_j$, then $HI=J_1K_1=J_2K_2$, as $TU\cap XYA_0F_j=\varnothing$, a contradiction. 

Assume that the $HI$ board contains a hyperedge of the star $\mathcal{E}_1 \cup \mathcal{E}_5$ taken by \textbf{P1}. Then the $HI$ board contains at most one hyperedge from the $TU$ board and hence it contains at most one hyperedge from $\mathcal{E}_0 \cup \mathcal{E}_2 \cup \mathcal{E}_3$, since $\mathcal{E}_0\cup\mathcal{E}_2\subset G^1$ is in the $TU$ board and so is $\mathcal{E}_3$ by definition. Notice that a star in one $2$-uniform board is a star in any $2$-uniform board, and hence the $HI$ board contains a star from $\mathcal{E}_1 \cup \mathcal{E}_5$. Finally, the $HI$ board contains at most two hyperedges from $\mathcal{E}_4 \cup \mathcal{E}_6$. Altogether, on the $HI$ board, \textbf{P1} has at most a star, and three extra hyperedges. However, by Remark~\ref{rem:conda}, this is not enough to complete a $G'$ -- contradiction.

Otherwise, assume that $G'$ has no hyperedge of the star $\mathcal{E}_1\cup\mathcal{E}_5$. Further, by the proof of Lemma~\ref{lem:2} we can assume that $\{H,I\}\neq\{T,U\}$. However, disregarding the star $\mathcal{E}_1\cup\mathcal{E}_5$, the only vertices with \textbf{P1}-degree at least $9$ are $T$, $U$, $C_0$ and $C_1$ and the last two have joint \textbf{P1}-degree at most $1+2$, since $\mathcal{E}_2\subset G^1$ is in the $TU$ board, which has only one hyperedge $C_0C_1$. So, without loss of generality, we can assume that $\{H,I\}=\{T,C_0\}$. Notice that the intersection of the $TU$ and $HI=TC_0$ boards is a star from $U$ when viewed in the $TC_0$ board and $\mathcal{E}_0\cup\mathcal{E}_2\subset G^1$ is in the $TU$ board and so is $\mathcal{E}_3$ by definition. Therefore, on the $TC_0$ board, \textbf{P1} has only a star (from $U$) from $\mathcal{E}_{0}\cup\mathcal{E}_2 \cup \mathcal{E}_3$ and at most $2$ extra hyperedges from $\mathcal{E}_4\cup\mathcal{E}_6$, which is not enough to complete a $G'$ by Remark~\ref{rem:conda} -- contradiction.\\

\noindent\textbf{Case II.} \textbf{P1} has at least $8$ hyperedges in a copy of $G'$ containing $TUVW$ with its base $C_0C_1$ absent.

\textbf{P2} takes $C_0C_1$ and then, in the fourth stage, proceeds as in Case I with the same proof.\\

\noindent\textbf{Case III.} For any $G'$ containing $TUVW$, $e_{\textbf{P1}}(G')\leq 7$.

In this case we directly skip to the fourth stage. \textbf{P2} takes hyperedges in the $XY$ board from $A_1$ to free vertices $F_i$ for $k<i\le l$ until at some point \textbf{P1} does not take the hyperedge $A_0F_l$. Then \textbf{P2} takes $A_0F_l$ and constructs a $G'$. The game stops after \textbf{P2}'s move. Let $\mathcal{E}_3=\{A_0F_i,k<i<l\}$ and $\mathcal{E}_4$ be the set consisting of the last hyperedge taken by \textbf{P1}. Assume for a contradiction that \textbf{P1} has a $G'$ with centres $HI$ and base $C_0C_1$. 

Note that $G'$ cannot contain $TUVW$, since then $e_{\textbf{P1}}(G')\leq 8$.  Also, $G'$ cannot be disjoint from both $TUVW$ and the star $\mathcal{E}_1 \cup \mathcal{E}_3$, as $e_{\textbf{P1}}(G')\le|(\mathcal{E}_0\setminus\{TUVW\})\cup\mathcal{E}_2\cup\mathcal{E}_4|\le 8$. Thus, $G'$ intersects the star $\mathcal{E}_1 \cup \mathcal{E}_3$, but does not contain $TUVW$. Recall that $\deg_\textbf{P1}(F_j)\leq 1+|\mathcal{E}_2\cup\mathcal{E}_4|\le 4$ implies that $F_j$ is not a centre nor is it in a base. Therefore, without loss of generality $\{H,I,C_0\}=\{X,Y, A_0\}$. In Lemma~\ref{lem:2} we already proved that $G'$ cannot be on the $XY$ board, so without loss of generality $(H,I,C_0)=(X,A_0,Y)$. Moreover $C_1 \neq A_1$ as $XYA_0A_1$ is taken by \textbf{P2}. Since without $\mathcal{E}_2\cup\mathcal{E}_4$ we have $\deg_\textbf{P1}(F_i)=1$ for all $i$, $G'$ has at most $|\mathcal{E}_2\cup\mathcal{E}_4|\le 3$ hyperedges from the star $\mathcal{E}_1\cup\mathcal{E}_3$. In addition, $G'$ has at most $3$ hyperedges from $\mathcal{E}_2\cup\mathcal{E}_4$. This means that $G'$ contains at least $3$ hyperedges from $\mathcal{E}_0$. Recall that by construction, there are at most $2$ hyperedges in $\mathcal{E}_0$ that contain $A_0$ but do not contain $A_1$. To arrive at the desired contradiction it is enough to show that $G'$ does not contain any hyperedge of the form $XA_0JA_1$. Indeed, if such a hyperedge is in $G'$, then necessarily $J=C_1$ or $J=C_0$, since it needs to intersect the base $C_0C_1$. However, the hyperedge $XA_0C_0A_1=XA_0YA_1$ is taken by \textbf{P2} and the hyperedge $XA_0C_1A_1$ cannot be in $G'$ because \textbf{P2} has $XA_0C_0A_1=XA_0YA_1$.
\end{proof}
\begin{rem}
One can easily extend this proof to obtain $r$-uniform graphs with the same property for all $r\geq 4$.
\end{rem}

\section{Concluding remarks}
\label{sec:remarks}
\subsection{Key aspects}
As the strategy is quite cumbersome and does not have a clear structure, it is not hard to miss the forest for the trees. Let us point out a few aspects which are important for the understanding of the dynamic of the Ramsey game in general.

The strategy reflects the significant differences between weak and strong games. The most important feature which arises in this context is the notion of delay: unlike the strategy in \cite{Hefetz17}, in ours \textbf{P2} does not focus solely on building a core fast, but rather needs to delay \textbf{P1} at least as much as \textbf{P1} delays him. We now point out some parts of the strategy which best highlight this feature.

Firstly, as it is clear from Figure~\ref{fig:tree}, \textbf{Case 1.} is much easier than \textbf{Case 2.}, so the second move of \textbf{P1} is crucially important. Unexpectedly, the harder \textbf{Case 2.} corresponds to two seemingly contradictory behaviours of \textbf{P1}. On the one hand, \textbf{P1} goes after \textbf{P2} in $K^2$, which can only have the purpose of blocking \textbf{P2}, since \textbf{P2} is the first player there. On the other hand, \textbf{P1} plays disjointly from \textbf{P2}'s edge, which is not the most natural of ``blocking'' moves. From another perspective, this is less surprising. Given that we prove exactly that \textbf{P2} can force a draw by moving away from \textbf{P1} in the other copy of $K_\omega$, it is reasonable to expect that the most efficient move of \textbf{P1} would be to also start playing ``away'' from \textbf{P2}.

Another striking fact is that the ``right'' move of \textbf{P2} in \textbf{Case \breakdot{2.1.}} is to go after \textbf{P1}, who is the ``second player'' in $K^2$. The purpose of this move is to constrain \textbf{P1}'s possibilities and later force \textbf{P1} to lose edges. This philosophy is even more visible in \textbf{Case \breakdot{2.1.2.1.1.}}. There \textbf{P2} takes an edge which is completely useless to him from a ``Maker'' perspective, but makes \textbf{P1} lose $2$ edges, so it is as though by playing this ``Breaker'' move \textbf{P2} gets a net advantage of $1$ edge. 

These examples show that a good strong game strategy should by all means seek to interact with the other player. However, the main approach to strong games so far relies on finding fast ``Maker'' strategies for the corresponding weak game and transforming them into strategies for the strong games \cite{Ferber11,Ferber14,Hefetz17}. Thus, our strategy suggests that the Ramsey game requires more sophisticated arguments.

On a more technical note, the idea of leaving additional edges unspecified and disregarding the stage of the game when an edge was played decreases the number of cases tremendously. This approach allows us to only have the \emph{tiny} tree in Figure~\ref{fig:tree}, which can readily be analysed without the help of a computer in short time. The driving force of the drawing strategy is Lemma~\ref{lem:2}, which provides us with intermediate desired configurations together with a drawing strategy from these configurations.

\subsection{Open problems}
\label{secopen}
The present work raises various natural closely related questions which, in view of our results, do not seem as widely open and intractable as previously.

The prime question to answer is the following.
\begin{que}
Is the game $\mathcal{R}(K_\omega,G)$ a draw?
\end{que}
As we already mentioned if the answer is positive, then this would be the first such known graph. Otherwise, $G$ would be the answer to the next question.
\begin{que}
Does there exist a target graph $H$ such that $\mathcal{R}(K_\omega,H)$ is a \textbf{P1}-win, but $\mathcal{R}(K_\omega\sqcup K_\omega,H)$ is not?
\end{que}
Such a counterexample would prove that yet another easy and folklore fact for finite boards breaks down for infinite ones.

Another line of thought, which might not lead directly towards $\mathcal{R}(K_n,K_k)$, was suggested by Bowler \cite{Bowler12}. He proved the natural characterisation of positional games which are a draw on infinite disjoint copies of a finite one. One can wonder if the game on $G$ gives a draw on an infinite number of copies of $K_n$. We believe that this is not the case.
\begin{conj}
For $n$ sufficiently large $\mathcal{R}(\bigsqcup_\omega K_n,G)$ is a \textbf{P1}-win.
\end{conj}
Finally, this last conjecture seems closely related to determining if the initiative in the sense of \cite{Bowler12} of $\mathcal{R}(K_n,K_4)$ is $6$, as asked in the next question.
\begin{que}
Let $n$ be a fixed large integer. What is the minimal $k$ such that \textbf{P1} has a winning strategy for $\mathcal{R}(K_n,K_4)$ such that if the second player is allowed to pass, he would not be able to do so more than $k$ times before losing?
\end{que}

\section*{Acknowledgements}
The authors would like to thank Imre Leader for useful discussions and the anonymous referees for suggesting improvements of the presentation. The third author would like to thank Trinity Hall, University of Cambridge, for support through the Trinity Hall Research Studentship.

\bibliographystyle{plain}
\bibliography{Bib}

\begin{thebibliography}{10}

\bibitem{Beck81}
J.~Beck.
\newblock Van der {W}aerden and {R}amsey type games.
\newblock {\em Combinatorica}, 1(2):103--116, 1981.

\bibitem{Beck82a}
J.~Beck.
\newblock Remarks on positional games. {I}.
\newblock {\em Acta Math. Acad. Sci. Hungar.}, 40(1-2):65--71, 1982.

\bibitem{Beck85}
J.~Beck.
\newblock Random graphs and positional games on the complete graph.
\newblock In {\em Random graphs '83 ({P}ozna\'n, 1983)}, volume 118 of {\em
  North-Holland Math. Stud.}, pages 7--13. North-Holland, Amsterdam, 1985.

\bibitem{Beck96}
J.~Beck.
\newblock Foundations of positional games.
\newblock In {\em Proceedings of the {S}eventh {I}nternational {C}onference on
  {R}andom {S}tructures and {A}lgorithms ({A}tlanta, {GA}, 1995)}, volume~9,
  pages 15--47, 1996.

\bibitem{Beck02}
J.~Beck.
\newblock Ramsey games.
\newblock {\em Discrete Math.}, 249(1-3):3--30, 2002.
\newblock Combinatorics, graph theory and computing (Louisville, KY, 1999).

\bibitem{Beck08}
J.~Beck.
\newblock {\em Combinatorial games}, volume 114 of {\em Encyclopedia of
  Mathematics and its Applications}.
\newblock Cambridge University Press, Cambridge, 2011.
\newblock Tic-tac-toe theory, Paperback edition of the 2008 original.

\bibitem{Beck82b}
J.~Beck and L.~Csirmaz.
\newblock Variations on a game.
\newblock {\em J. Combin. Theory Ser. A}, 33(3):297--315, 1982.

\bibitem{Berge76}
C.~Berge.
\newblock Sur les jeux positionnels.
\newblock {\em Cahiers Centre \'Etudes Recherche Op\'er.}, 18(1-2):91--107,
  1976.
\newblock Colloque sur la Th\'eorie des Jeux (Brussels, 1975).

\bibitem{Bowler12}
N.~Bowler.
\newblock Winning an infinite combination of games.
\newblock {\em Mathematika}, 58(2):419--431, 2012.

\bibitem{Erdos73}
P.~Erd\H{o}s and J.~L. Selfridge.
\newblock On a combinatorial game.
\newblock {\em J. Combinatorial Theory Ser. A}, 14:298--301, 1973.

\bibitem{Ferber11}
A.~Ferber and D.~Hefetz.
\newblock Winning strong games through fast strategies for weak games.
\newblock {\em Electron. J. Combin.}, 18(1):Paper 144, 13, 2011.

\bibitem{Ferber14}
A.~Ferber and D.~Hefetz.
\newblock Weak and strong {$k$}-connectivity games.
\newblock {\em European J. Combin.}, 35:169--183, 2014.

\bibitem{Hales63}
A.~W. Hales and R.~I. Jewett.
\newblock Regularity and positional games.
\newblock {\em Trans. Amer. Math. Soc.}, 106:222--229, 1963.

\bibitem{Harary72}
F.~Harary.
\newblock Recent results on generalized {R}amsey theory for graphs.
\newblock pages 125--138. Lecture Notes in Math., Vol. 303, 1972.

\bibitem{Harary82}
F.~Harary.
\newblock Achievement and avoidance games for graphs.
\newblock In {\em Graph theory ({C}ambridge, 1981)}, volume~13 of {\em Ann.
  Discrete Math.}, pages 111--119. North-Holland, Amsterdam-New York, 1982.

\bibitem{Hefetz17}
D.~Hefetz, C.~Kusch, L.~Narins, A.~Pokrovskiy, C.~Requilé, and A.~Sarid.
\newblock Strong {R}amsey games: Drawing on an infinite board.
\newblock {\em Journal of Combinatorial Theory, Series A}, 150:248 -- 266,
  2017.

\bibitem{Pekec96}
A.~Peke\v{c}.
\newblock A winning strategy for the {R}amsey graph game.
\newblock {\em Combin. Probab. Comput.}, 5(3):267--276, 1996.

\bibitem{Ramsey30}
F.~P. Ramsey.
\newblock On a problem of formal logic.
\newblock {\em Proc. London Math. Soc.}, S2-30(1):264, 1930.

\end{thebibliography}

\end{document}